\documentclass[10pt,reqno]{article}

\usepackage[b5paper,top=1.2in,left=0.9in]{geometry}
\usepackage{verbatim}
\usepackage{amssymb}
\usepackage{amsmath}
\usepackage{graphicx}
\usepackage{appendix}
\usepackage{color}
\usepackage{amsthm}
\usepackage{tikz}
\usetikzlibrary{arrows,positioning,decorations.pathmorphing,  decorations.markings}

\renewcommand{\d}{\mathrm{d}}
\newcommand{\D}{\mathrm{D}}

\newcommand{\e}{\mathrm{e}}

\newtheorem{Thm}{Theorem}[section]
\newtheorem{Lem}[Thm]{Lemma}
\newtheorem{Prop}[Thm]{Proposition}

\newtheorem{Rem}[Thm]{Remark}
\newtheorem{Def}[Thm]{Definition}

\newtheorem{Ex}[Thm]{Example}
\newtheorem{Fact}[Thm]{Fact}
\newtheorem{Nota}[Thm]{Notation}

\newtheoremstyle{named}{}{}{\itshape}{}{\bfseries}{.}{.5em}{#1 #3}
\theoremstyle{named}
\newtheorem*{MThm}{Theorem}

\newtheorem*{MDef}{Definition}

\def\R{\mathbb{R}}
\def\Q{\mathbb{Q}}

\def\C{\mathbb{C}}
\def\Z{\mathbb{Z}}
\def\T{\mathbb{T}}

\def\to{\longrightarrow}

\def\cP{\mathcal{P}}

\def\cU{\mathcal{U}}

\def\a{\alpha}

\def\e{\epsilon}

\def\c{\gamma}
\def\D{\Delta}

\def\d{\delta}

\def\l{\lambda}

\def\W{\Omega}
\def\w{\omega}
\def\ze{\zeta}

\def\sl{\mathfrak{sl}}
\def\g{\mathfrak{g}}

\def\fh{\mathfrak{h}}
\def\fq{\mathfrak{q}}

\def\ox{\otimes}
\def\o+{\oplus}
\def\bo+{\bigoplus}
\def\x{\times}
\def\p[#1,#2]{\phi_{#1,#2}}
\def\til[#1]{\widetilde{#1}}
\def\what[#1]{\widehat{#1}}

\def\bb{\textbf{b}}

\def\bu{\textbf{u}}
\def\bv{\textbf{v}}

\def\be{\textbf{e}}
\def\bE{\textbf{E}}
\def\bf{\textbf{f}}
\def\bF{\textbf{F}}

\def\bi{\textbf{i}}

\def\bK{\textbf{K}}

\def\bU{\textbf{U}}

\def\z[#1]{z_{#1}}

\def\=>{\Longrightarrow}
\def\inj{\hookrightarrow}

\def\<{\langle}
\def\>{\rangle}
\def\corr{\longleftrightarrow}
\def\^{\wedge}
\def\+{\dagger}

\def\iff{\Longleftrightarrow}

\def\inv{^{-1}}

\def\over[#1]{\overline{#1}}
\def\vec[#1]{\overrightarrow{#1}}
\def\mat[#1, #2]{\left[\begin{array}{ccccc}#1\end{array}\left|\begin{array}{c}#2\end{array}\right.\right]}
\def\xto[#1]{\xrightarrow{#1}}
\def\dd[#1,#2]{\frac{d#1}{d#2}}
\def\del[#1,#2]{\frac{\partial #1}{\partial #2}}
\def\Facts[#1]{\begin{Fact}\mbox{}\begin{itemize}#1\end{itemize}\end{Fact}}
\def\Notation[#1]{\begin{Nota}\mbox{}\begin{itemize}#1\end{itemize}\end{Nota}}
\def\Eqn[#1]{\begin{eqnarray*}#1\end{eqnarray*}}
\def\tab{\;\;\;\;\;\;}

\newcommand{\veca}[2][cccccccccccccccccccccccccccccccccccccccccc]{\left(\begin{array}{#1}#2 \\ \end{array} \right)}

\newcommand{\Eq}[1]{\begin{align}#1\end{align}}
\newcommand{\case}[2][cccccccccccccccccccccccccccccccccccccccccc]{\left\{\begin{array}{#1}#2 \\ \end{array}\right.}

\begin{document}
\title{Positive representations of non-simply-laced split~real quantum groups}

\author{  Ivan C.H. Ip\footnote{
         	Kavli Institute for the Physics and Mathematics of the Universe (WPI), 
		The University of Tokyo, 
		Kashiwa, Chiba 
		277-8583, Japan
		\newline
		Email: ivan.ip@ipmu.jp
          }
}

\date{\today}

\numberwithin{equation}{section}

\maketitle

\begin{abstract}
We construct the positive principal series representations for $\cU_q(\g_\R)$ where $\g$ is of type $B_n$, $C_n$, $F_4$ or $G_2$, parametrized by $\R^n$ where $n$ is the rank of $\g$. We show that under the representations, the generators of the Langlands dual group $\cU_{\til[q]}({}^L\g_\R)$ are related to the generators of $\cU_q(\g_\R)$ by the transcendental relations. This gives a new and very simple analytic relation between the Langlands dual pair. We define the modified quantum group $\bU_{\fq\til[\fq]}(\g_\R)=\bU_{\fq}(\g_\R)\ox\bU_{\til[\fq]}({}^L\g_\R)$ of the modular double and show that the representations of both parts of the modular double commute with each other, and there is an embedding into the $q$-tori polynomials. 
\end{abstract}

{\small {\textbf{Keywords.} positive representations; non-simply-laced; split real quantum groups; Langlands dual; modular double}
\\

{\small {\textbf{2010 MSC:} 17B37, 81R50}
\newpage
\tableofcontents
\section{Introduction}\label{sec:intro}
In this paper, we give the construction of the positive principal series representations for the quantum group $\cU_{q}(\g_\R)$ and its modular double where $\g$ is of non-simply-laced type $B_n$, $C_n$, $F_4$ or $G_2$, generalizing our recent work \cite{Ip} on the simply-laced case, thus completing the constructions corresponding to simple $\g$ of all types. The transcendental relations that were part of the axioms in the simply-laced case, relates the quantum group $\cU_q(\g_\R)$ with its Langlands dual $\cU_{\til[q]}({}^L\g_\R)$ in the non simply-laced case. This might be considered as the simplest realization of the Langlands dual pair, given by a single analytic relation.

The notion of the \emph{positive principal series representations} was introduced in \cite{FI} as a new research program devoted to the representation theory of split real quantum groups. It uses the concept of modular double for quantum groups \cite{Fa}, and has been studied for $\cU_{q\til[q]}(\sl(2,\R))$ by Ponsot and Teschner \cite{PT1}. Let us recall the definition in the simply-laced case. Let $E_i,F_i,K_i$ be the generators of $\cU_q(\g_\R)$ with the standard quantum relations, where $q=e^{\pi ib^2}$, $b^2\in\R\setminus\Q$ and $0<b<1$. Similarly let $\til[E_i],\til[F_i],\til[K_i]$ be the generators of $\cU_{\til[q]}(\g_\R)$ by replacing $b$ with $b\inv$, where $\til[q]=e^{\pi ib^{-2}}$. Then using the rescaled variables
\Eq{e_i=2\sin(\pi b^2)E_i, \tab f_i=2\sin(\pi b^2)F_i} and similarly for $\til[e_i]$ and $\til[f_i]$ with $b$ replaced by $b\inv$,
the positive representations has the following remarkable properties:
\begin{itemize}
\item[(i)] the generators $e_i,f_i,K_i^{\pm 1}$ and $\til[e_i],\til[f_i],\til[K_i]^{\pm 1}$ are represented by positive essentially self-adjoint operators,
\item[(ii)] the generators satisfy the transcendental relations
\Eq{\label{trans}e_i^{\frac{1}{b^2}}=\til[e_i],\tab f_i^{\frac{1}{b^2}}=\til[f_i],\tab K_i^{\frac{1}{b^2}}=\til[K_i].}
\end{itemize}
Furthermore, by modifying the definition of $e_i, f_i, K_i^{\pm1}$ and the tilde variables with certain factors of $K_i$'s, we also obtain the compatibility with the modular double $\bU_{\fq\til[\fq]}(\g_\R)$:
\begin{itemize}
\item[(iii)]  the generators $\be_i,\bf_i,\bK_i^{\pm 1}$ commute with $\til[\be_i],\til[\bf_i],\til[\bK_i]^{\pm 1}$.
\end{itemize}

Note that since the generators are represented by positive operators, the real powers $\frac{1}{b_i^2}$ defining the transcendental relations are well-defined by means of functional calculus. However, when $\g$ is not simply-laced, the transcendental relations \eqref{trans} is substantially different. The transcendental relations now relate $\g$ with its Langlands dual ${}^L\g$ directly with appropriate changes of parameters. This is a special feature of the non-simply-laced case for quantum groups. The starting point of this paper is the following remarkable result:

\begin{Thm}[Langlands duality for quantum group]\label{MainThm}
For each positive simple root $\a_i$, define $$q_i=q^{\frac{1}{2}(\a_i,\a_i)}=e^{\pi ib_i^2},$$ and let $b_s=b_i$ when $\a_i$ is a short root (see Definition \ref{qi}). Let $\til[q]=e^{\pi ib_s^{-2}}$. Define the operators
\begin{eqnarray}
\til[e_i]&:=&(e_i)^{\frac{1}{b_i^2}},\\
\til[f_i]&:=&(f_i)^{\frac{1}{b_i^2}},\\
\til[K_i]&:=&(K_i)^{\frac{1}{b_i^2}},
\end{eqnarray}
where
\begin{eqnarray}
e_i=2\sin\pi b_i^2E_i,&\tab& f_i=2\sin\pi b_i^2F_i,\\
\til[e_i]=2\sin\pi b_i^{-2}\til[E_i],&\tab&\til[f_i]=2\sin\pi b_i^{-2}\til[F_i].
\end{eqnarray}
Then the generators of $\cU_{\til[q]}({}^L\g_\R)$ are represented by the operators $\til[E_i],\til[F_i]$ and $\til[K_i]$, where ${}^L\g_\R$ is defined by replacing long roots with short roots and short roots with long roots in the Dynkin diagram of $\g$.
\end{Thm}

We remark that the Langlands dual already in the simply-laced case appear as the commutant of the quantum group $\cU_q(\g_\R)$ \cite[Theorem 10.4]{Ip}, where the type of algebra is the same. The possibility of such phenomenon has been suggested in \cite{GKL}. Hence Theorem \ref{MainThm} about the transcendental relations gives an explicit construction of the Langlands dual in the context of representation theory of split real quantum groups, where in the non-simply-laced case the type of algebra is different. Furthermore it cannot be obtained in the classical setting as $b\to 0$. This relation between modular duality and Langlands duality should indeed have deep consequences, as pointed out for example in \cite{FG,T}.

With the above theorem, we define the (modified) modular double by
\Eq{\bU_{\fq\til[\fq]}(\g_\R):=\bU_{\fq}(\g_\R)\ox \bU_{\til[\fq]}({}^L\g_\R),}
and the main results of the paper are the following:
\begin{Thm}There exists a family of positive principal series representation for $\cU_q(\g_\R)$ and its (modified) modular double $\bU_{\fq\til[\fq]}(\g_\R)$, parametrized by $\l\in\R_{\geq0}^n$ where $n=rank(\g)$, satisfying properties (i),(iii) and Theorem \ref{MainThm} above.
\end{Thm}

More precisely, for every reduced expression for $w_0$, we parametrize $U_{>0}^+$ using the Lusztig's coordinate, and construct explicitly the positive representations. For each change of words of $w_0$, we establish a unitary transformation (Theorem \ref{trans2}), so that in particular the family of positive representation is \emph{independent of choice} of reduced expression of $w_0$. 

Hence by choosing a ``good" reduced expression for $w_0$, we can write down explicitly the positive representations for $\cU_q(\g_\R)$. Not surprisingly, by the philosophy of folding of Dynkin diagram, we observe the following
\begin{Thm}The positive representations of $\cU_q(\g_\R)$ of type $B_n$, $C_n$, $F_4$ and $G_2$ can be obtained essentially (up to some quantized multiples) from the positive representations of $\cU_q(\g_\R)$ of type $A_{2n-1}$, $D_{n+1}$, $E_6$ and $D_4$ respectively under certain identifications of roots.
\end{Thm}

Finally, as in the simply-laced case, by using the modified version $\bU_{\fq\til[\fq]}(\g_\R)$ of the modular double, we have the following properties.
\begin{Thm} \label{ThmLL}The commutant of (the adjoint form) $\bU_{\fq}(\g_\R)$ is the simply-connected form of the Langlands dual group $\widehat{\bU_{\til[\fq]}}({}^L\g_\R)$.
\end{Thm}
In fact, from the proof of Theorem \ref{ThmLL}, the correspondence of the commutant of the quantum groups corresponding to other forms in between the simply-connected and the adjoint forms can be deduced.

\begin{Thm}\label{Thmqtori}Let $s$ (resp. $l$) be the number of indices corresponding to short (resp. long) roots, so that $s+l=l(w_0)$. Then we have an embedding
\Eq{\bU_{\fq\til[\fq]}(\g_\R)\inj \C[\T_{\fq\til[\fq]}^{s,l}].}
of the modified modular double into the Laurent polynomials generated by $s$ $q_s$-tori and $l$ $q_l$-tori and their modular double counterparts. In particular each generator of $\bU_{\fq\til[\fq]}(\g_\R)$ is realized as a Laurent polynomial of the $q$-tori variables.
\end{Thm}

\begin{Thm}\label{RW} The positive representations corresponding to the parameters $\l$ and $w(\l)$ are unitary equivalent,  where  $\l\in\R^n$, and $w\in W$ is a Weyl group element. In particular the positive representations are parametrized by $\l\in\R_{\geq 0}^n$.
\end{Thm}

The paper is organized as follows. In Section \ref{sec:prelim}, we fix some notations and recall the definition of $\cU_q(\g_\R)$ of general type, the definition of the quantum dilogarithm function, and Lusztig's parametrization of the positive unipotent semi-subgroup $U_{>0}^+$ of $G$. In Section \ref{sec:construct}, we give the general construction of the positive representations for $F_i$ and $K_i$, and the action of $E_i$ for a particular choice of $w_0$. In Section \ref{sec:B2}, we study in detail the positive representations of $\cU_q(\g_\R)$ where $\g$ is of type $B_2$, and describe the transformation needed to relate different reduced expression of $w_0$. In Section \ref{sec:posrep} and \ref{sec:G2} we give the explicit action of the positive representations of all type. In Section \ref{sec:folding} we describe the relations between the positive representations and folding of Dynkin diagram. In Section \ref{sec:langlands}, we prove the main theorem about the transcendental relations which is related to the Langlands dual quantum group. Finally in Section \ref{sec:modified} we introduce the modified quantum group $\bU_{\fq\til[\fq]}(\g_\R)$ and state the main theorems about the positive representations of the modular double, the Langlands dual as the commutant, and its embedding into the $q$-tori.

\textbf{Acknowledgments.} This work was partially supported by Yale University and World Premier International Research Center Initiative (WPI Initiative), MEXT, Japan.

\section{Preliminaries}\label{sec:prelim}
Throughout the paper, we will let $q=e^{\pi ib^2}$ with $0<b^2<1$ and $b\in\R\setminus\Q$.

\subsection{Definition of $\cU_q(\g_\R)$}\label{sec:Uq}

We recall the definition of the quantum group $\cU_q(\g_\R)$ where $\g$ is of general type \cite{CP,Lu}. 

\begin{Def} Let $I$ denote the set of nodes of the Dynkin diagram of $\g$, with the following labeling. Here the black nodes correspond to short roots, and white nodes correspond to long roots.

The Dynkin diagram for Type $B_n$ is given by
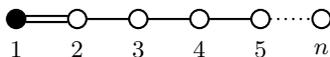
\begin{center}
  \begin{tikzpicture}[scale=.4]
    \draw[xshift=0 cm,thick,fill=black] (0 cm, 0) circle (.3 cm);
    \foreach \x in {1,...,5}
    \draw[xshift=\x cm,thick] (\x cm,0) circle (.3cm);
    \draw[dotted,thick] (8.3 cm,0) -- +(1.4 cm,0);
    \foreach \y in {1.15,...,3.15}
    \draw[xshift=\y cm,thick] (\y cm,0) -- +(1.4 cm,0);
    \draw[thick] (0.3 cm, .1 cm) -- +(1.4 cm,0);
    \draw[thick] (0.3 cm, -.1 cm) -- +(1.4 cm,0);
    \foreach \z in {1,...,5}
    \node at (2*\z-2,-1) {$\z$};
\node at (10,-1){$n$};
  \end{tikzpicture}
\end{center}
and the corresponding Cartan matrix is given by ($1\leq i,j\leq n$):
\Eq{a_{ij}=\case{2&i=j,\\-2 &(i,j)=(1,2),\\-1 &\mbox{$|i-j|=1$ and $(i,j)\neq (1,2),$}\\0&\mbox{ otherwise.}}}

The Dynkin diagram for Type $C_n$ is given by
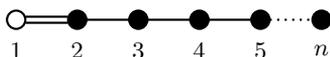
\begin{center}
  \begin{tikzpicture}[scale=.4]
    \draw[xshift=0 cm,thick] (0 cm, 0) circle (.3 cm);
    \foreach \x in {1,...,5}
    \draw[xshift=\x cm,thick,fill=black] (\x cm,0) circle (.3cm);
    \draw[dotted,thick] (8.3 cm,0) -- +(1.4 cm,0);
    \foreach \y in {1.15,...,3.15}
    \draw[xshift=\y cm,thick] (\y cm,0) -- +(1.4 cm,0);
    \draw[thick] (0.3 cm, .1 cm) -- +(1.4 cm,0);
    \draw[thick] (0.3 cm, -.1 cm) -- +(1.4 cm,0);
    \foreach \z in {1,...,5}
    \node at (2*\z-2,-1) {$\z$};
\node at (10,-1){$n$};
  \end{tikzpicture}
\end{center}
and the corresponding Cartan matrix is given by ($1\leq i,j\leq n$):
\Eq{a_{ij}=\case{2&i=j,\\-2&(i,j)=(2,1),\\-1 &\mbox{$|i-j|=1$ and $(i,j)\neq (2,1)$},\\0&\mbox{ otherwise.}}}

The Dynkin diagram for Type $F_4$ is given by

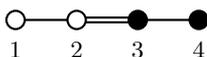
\begin{center}
  \begin{tikzpicture}[scale=.4]
    \draw[thick] (-2 cm ,0) circle (.3 cm);
	\node at (-2,-1) {$1$};
    \draw[thick] (0 ,0) circle (.3 cm);
	\node at (0,-1) {$2$};
    \draw[thick,fill=black] (2 cm,0) circle (.3 cm);
	\node at (2,-1) {$3$};
    \draw[thick,fill=black] (4 cm,0) circle (.3 cm);
	\node at (4,-1) {$4$};
    \draw[thick] (15: 3mm) -- +(1.5 cm, 0);
    \draw[xshift=-2 cm,thick] (0: 3 mm) -- +(1.4 cm, 0);
    \draw[thick] (-15: 3 mm) -- +(1.5 cm, 0);
    \draw[xshift=2 cm,thick] (0: 3 mm) -- +(1.4 cm, 0);
  \end{tikzpicture}
\end{center}
and the corresponding Cartan matrix is given by
\Eq{A=(a_{ij})=\veca{2&-1&0&0\\-1&2&-1&0\\0&-2&2&-1\\0&0&-1&2}.}
where $1,2$ are short roots, $3,4$ are long roots.

The Dynkin diagram for Type $G_2$ is given by
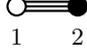
\begin{center}
  \begin{tikzpicture}[scale=.4]
    \draw[thick] (0 ,0) circle (.3 cm);
	\node at (0,-1) {$1$};
    \draw[thick,fill=black] (2 cm,0) circle (.3 cm);
	\node at (2,-1) {$2$};
    \draw[thick] (30: 3mm) -- +(1.5 cm, 0);
    \draw[thick] (0: 3 mm) -- +(1.4 cm, 0);
    \draw[thick] (-30: 3 mm) -- +(1.5 cm, 0);
  \end{tikzpicture}
\end{center}
and the corresponding Cartan matrix is given by
\Eq{A=(a_{ij})=\veca{2&-1\\-3&2}.}
\end{Def}

\begin{Def} \label{qi} Let $(-,-)$ be the $W$-invariant inner product of the root lattice such that $(\a,\a)=2$ for long roots $\a$, where $W$ is the Weyl group of the Cartan datum. Let $\a_i$, $i\in I$ be the positive simple roots, and we define
\begin{eqnarray}
q_i=q^{\frac{1}{2}(\a_i,\a_i)}.
\end{eqnarray}
In the case when $\g$ is of type $B_n$, $C_n$ and $F_4$, we define $b_l=b$, and $b_s=\frac{b}{\sqrt{2}}$. with the following normalization:
\begin{eqnarray}
\label{qiBCF}q_i=\case{e^{\pi ib_l^2}=q&\mbox{$i$ is long root,}\\e^{\pi ib_s^2} =q^{\frac{1}{2}}&\mbox{$i$ is short root.}}
\end{eqnarray}
In the case when $\g$ is of type $G_2$, we define $b_l=b$, and $b_s=\frac{b}{\sqrt{3}}$ with the following normalization:
\begin{eqnarray}
\label{qiG}q_i=\case{e^{\pi ib_l^2}=q&\mbox{$i$ is long root,}\\e^{\pi ib_s^2} =q^{\frac{1}{3}}&\mbox{$i$ is short root.}}
\end{eqnarray}
\end{Def}

\begin{Def} Let $A=(a_{ij})$ denote the Cartan matrix. Then $\cU_q(\g_\R)$ with $q=e^{\pi ib_l^2}$ is generated by $E_i$, $F_i$ and $K_i^{\pm1}$, $i\in I$ subject to the following relations:
\begin{eqnarray}
K_iE_j&=&q_i^{a_{ij}}E_jK_i,\\
K_iF_j&=&q_i^{-a_{ij}}F_jK_i,\\
{[E_i,F_j]} &=& \d_{ij}\frac{K_i-K_i\inv}{q_i-q_i\inv},
\end{eqnarray}
together with the Serre relations for $i\neq j$:
\begin{eqnarray}
\sum_{n=0}^{1-a_{ij}}(-1)^n\frac{[1-a_{ij}]_{q_i}!}{[1-a_{ij}-n]_{q_i}![n]_{q_i}!}E_i^{n}E_jE_i^{1-a_{ij}-n}&=&0,\label{SerreE}\\
\sum_{n=0}^{1-a_{ij}}(-1)^n\frac{[1-a_{ij}]_{q_i}!}{[1-a_{ij}-n]_{q_i}![n]_{q_i}!}F_i^{n}F_jF_i^{1-a_{ij}-n}&=&0,\label{SerreF}
\end{eqnarray}
where $[n]_q=\frac{q^n-q^{-n}}{q-q\inv}$. We also define formally the elements $H_i$ so that $K_i=q_i^{H_i}$.
\end{Def}

\subsection{Quantum dilogarithm}\label{sec:qdlog}
Let us briefly recall the definition and some properties of the quantum dilogarithm functions \cite{Ip}. Let $Q:=b+b\inv$.

\begin{Def} The quantum dilogarithm function $G_b(x)$ is defined on\\ ${0\leq Re(z)\leq Q}$ by
\Eq{\label{intform} G_b(x)=\over[\ze_b]\exp\left(-\int_{\W}\frac{e^{\pi tz}}{(e^{\pi bt}-1)(e^{\pi b\inv t}-1)}\frac{dt}{t}\right),}
where \Eq{\ze_b=e^{\frac{\pi i}{2}(\frac{b^2+b^{-2}}{6}+\frac{1}{2})},}
and the contour goes along $\R$ with a small semicircle going above the pole at $t=0$. This can be extended meromorphically to the whole complex plane.
\end{Def}
\begin{Def} The function $g_b(x)$ is defined by
\Eq{g_b(x)=\frac{\over[\ze_b]}{G_b(\frac{Q}{2}+\frac{\log x}{2\pi ib})}.}
where $\log$ takes the principal branch of $x$.
\end{Def}
The function $g_b(x)$, also called the quantum dilogarithm, is the crucial tool for all the transformations between self-adjoint operators. In particular it gives the unitary transformation that relates the positive representations corresponding to different expressions of the longest element. In particular, we will need the following two properties of $g_b(x)$.

\begin{Lem}\cite{BT}\label{unitary} $|g_b(x)|=1$ when $x\in\R_+$, hence $g_b(X)$ is a unitary operator for any positive operator $X$.
\end{Lem}

\begin{Lem}\cite{BT}\label{qsum}Let $u,v$ be positive essentially self-adjoint operators. If $uv=q^2vu$, then
\begin{eqnarray}
\label{qsum1}g_b(u)^*vg_b(u) &=& q\inv uv+v,\\
\label{qsum2}g_b(v)ug_b(v)^*&=&u+q\inv uv.
\end{eqnarray}
If $uv=q^4vu$, then we apply the Lemma twice and obtain
\begin{eqnarray}
\label{qsum3}g_b(u)^*vg_b(u)&=&v+[2]_qq^{2}vu+q^{4}vu^2,\\
\label{qsum4}g_b(v)ug_b(v)^*&=&u+[2]_qq^{-2}uv+q^{-4}uv^2.
\end{eqnarray}
More generally, if $uv=q^{2n}vu$, then
\begin{eqnarray}
g_b(u)^*vg_b(u)&=&\sum_{k=0}^{n} \frac{[n]_q!}{[n-k]_q![k]_q!} q^kvu^{2k},\\
g_b(v)ug_b(v)^*&=&\sum_{k=0}^{n} \frac{[n]_q!}{[n-k]_q![k]_q!} q^{-k}uv^{2k}.\\
\end{eqnarray}
\end{Lem}

As a consequence of the above Lemma, we also have Volkov's magic lemma:
\begin{Lem}\cite{Vo}\label{qbi} If $uv=q^2vu$ where $u,v$ are positive essentially self-adjoint operators, then $u+v$ is also a positive essentially self-adjoint operator, and
\Eq{\label{qbiEq} (u+v)^{\frac{1}{b^2}}=u^{\frac{1}{b^2}}+v^{\frac{1}{b^2}}.}
\end{Lem}


\subsection{Lusztig's data}\label{sec:Ldata}
The following are described in detail in \cite{Lu2}. For simplicity, let $G$ be a simple Lie group. Recall that for any simple root $\a_i\in \D$ there exists a homomorphism $SL_2(\R)\to G$ denoted by
\begin{eqnarray}
\veca{1&a\\0&1}&\mapsto& x_i(a)\in U_i^+,\\
\veca{b&0\\0&b\inv}&\mapsto &\chi_i(b)\in T,\\
\veca{1&0\\c&1}&\mapsto &y_i(c)\in U_i^-,
\end{eqnarray}
called the \emph{pinning} of $G$, where $T$ is the split real maximal torus of $G$, and $U_i^+$ and $U_i^-$ are the simple root subgroups of $U^+$ and $U^-$ respectively. Then the positive unipotent semigroup $U_{>0}^+$ is defined by the image of the map $\R_{>0}^m\to U^+$ given by
\Eq{\label{w0coord}(a_1,a_2,...,a_m)\mapsto x_{i_1}(a_1)x_{i_2}(a_2)...x_{i_n}(a_m),}
where $s_{i_1}s_{i_2}...s_{i_m}$ is a reduced expression for the longest element $w_0$ of the Weyl group $W$. We define $U_{>0}^-$ in a similar way.

\begin{Lem}\cite{Lu2} We have the following identities:
\begin{eqnarray}
\label{EK}\chi_i(b)x_i(a)&=&x_i(b^2 a)\chi_i(b),\\
\label{EF}x_i(a)y_j(c)&=&y_j(c)x_i(a)\tab \mbox{ if $i\neq j$},\\
\label{EKF}x_i(a)\chi_i(b)y_i(c)&=&y_i(\frac{c}{ac+b^2})\chi_i(\frac{ac+b^2}{b})x_i(\frac{a}{ac+b^2}).
\end{eqnarray}

In the simply-laced case, assume the roots $\a_i$ and $\a_j$ are joined by an edge in the Dynkin diagram. Then we have

\begin{eqnarray}
\label{EK2}\chi_i(b)x_j(a)&=&x_j(b\inv a)\chi_i(b),\\
\label{121}x_i(a)x_j(b)x_i(c)&=&x_j(\frac{bc}{a+c})x_i(a+c)x_j(\frac{ab}{a+c}).
\end{eqnarray}
\end{Lem}

\section{Construction of positive representations}\label{sec:construct}
Let us recall the classical construction.
\begin{Prop} Let $\C[U_{>0}^+]$ be the space of continuous functions on the positive unipotent semigroup $U_{>0}^+$ defined in \eqref{w0coord}. The minimal principal series representation for $\cU(\g_\R)$ can be realized as the infinitesimal action of $g\in G_\R$ acting on $\C[U_{>0}^+]$ by
\Eq{\label{groupaction}g\cdot f(h) =\chi_\l(hg) f([hg]_+).}
Here we write the Gauss decomposition of $g$ as
\Eq{g=g_-g_0g_+\in U_{>0}^-T_{>0}U_{>0}^+,}
so that $[g]_+=g_+$ is the projection of $g$ onto $U_{>0}^+$, and $\chi_\l(g)$ is the character function defined by
\Eq{\chi_\l(g)=\prod_{i=1}^n u_i^{2\l_i},}
where $n$ is the rank of $\g_\R$, $\l=(\l_i)\in \C^n$ and $u_i=\chi_i\inv(g_0)\in T_{>0}$. (One can also treat $\l:=\sum _{i=1}^n\l_i\a_i^{\vee}\in\fh^*$ where $\a_i^{\vee}$ are the dual coroots.)
\end{Prop}
Following \cite{Ip}, we will use the formal Mellin transformation of the form
\Eq{\label{MellinTransform}f(u):=\int F(x)x^u dx}
 on each variable, which transforms differential operators on $F(x)$ into finite difference operators on $f(u)$. Using this technique, the positive representations in the simply-laced case are constructed in \cite{Ip}. We extend the construction to all types as follows.

\begin{Def}\label{labelling}Let us denote the Lusztig's coordinates of $U_{>0}^+$ given in Section \ref{sec:Ldata} by $x_i^k$, where $i$ is the corresponding root index, and $k$ denotes the sequence this root is appearing in $w_0$ from the right. Similarly we denote by $u_i^k$ the Mellin transformed variables. We will also denote the Mellin transformed variables by $v_i$, $1\leq i\leq l(w_0)$ counting from the left, and let $v(i,k)$ be the index such that $u_i^k=v_{v(i,k)}$.
\end{Def}

\begin{Ex} The coordinates for $A_3$ corresponding to $w_0=s_3s_2s_1s_3s_2s_3$ is given by
$$(u_3^3,u_2^2,u_1^1,u_3^2, u_2^1,u_3^1)=(v_1,v_2,v_3,v_4,v_5,v_6)$$
\end{Ex} 

Let $q_i=e^{\pi ib_i^2}$ and $Q_i=b_i+b_i\inv$ (cf. Definition \ref{qi}). We define the quantized action from the classical Mellin transformed action with the appropriate $q_i$-number. The quantized actions will then be unbounded, positive essentially self-adjoint operators acting on the Hilbert space $L^2(\R^{dim(U^+)})$.

\begin{Def} \label{E}Choose the reduced expression for $w_0=w_{l-1}s_i$ where $w_{l-1}$ is the reduced expression for $ws_i$. Then the classical (Mellin transformed) action of $E_i$ is given by
\Eq{E_i:f\mapsto (u_i^1+1)f(u_i^1+1),}
and we define the positive quantized action by
\begin{eqnarray}
E_i&=&\left[\frac{Q_i}{2b_i}-\frac{i}{b}u_i^1\right]_{q_i}e^{-2\pi bp_i^1}.
\end{eqnarray}
\end{Def}

For the operators $F_i$ and $H_i$, we slightly modify the action from \cite{Ip} so that the following hold for $\g$ of all type with Cartan matrix $(a_{ij})$.

\begin{Def}\label{FH}Let $r(j)$ be the root corresponding to the variable $v_j$. For any reduced expression $w_0$, the classical Mellin transformed action is given by
\Eq{F_i:f\mapsto\sum_{k=1}^n\left(1-\sum_{j=1}^{v(i,k)-1} a_{i,r(j)}v_j-u_i^k+2\l_i\right)f(u_i^k-1)}
and the positive quantized action is given by (with $\l_i\in\R$):
\begin{eqnarray}
F_i&=&\sum_{k=1}^n\left[\frac{Q_i}{2b_i}+\frac{i}{b}\left(\sum_{j=1}^{v(i,k)-1} a_{i,r(j)}v_j+u_i^k+2\l_i\right)\right]_{q_i}e^{2\pi bp_{i}^{k}}\label{FF}.
\end{eqnarray}

The classical Mellin transformed action of $H_i$ is multiplication by
\Eq{\label{HH}H_i= \sum_{j=1}^{l(w_0)} -a_{i,r(j)}v_j-2\l_i,}
and the quantized action (after rescaling) is given by
\Eq{\label{KK}K_i=q_i^{H_i}=e^{-\pi b_i(\sum_{k=1}^{l(w_0)} a_{i,r(k)}v_k+2\l_i)}.}
\end{Def}

From these definitions, for each fixed choice of expression of $w_0$, the commutation relations between $E_i, F_i$ and $K_i$ are easily checked. The Serre relations between $E_i$ will be given by the explicit rank 2 expressions in the next section and the unitary transformations of operators. For the action of $F_i$ we have:
\begin{Thm}\label{Serre} The action of $F_i$ defined above satisfy the quantum Serre relations \eqref{SerreF}.
\end{Thm}
\begin{proof} The method of the proof is similar to the approach givein in \cite{FI}.  Let $i$ corresponds to the short root and $j$ the long root, and let us write $$F_i=\sum_{k=1}^n [F_i^k(\bv)]_{q_i}e^{2\pi bp_i^k},$$ where $F_i^k(\bv)$ are linear functions in $\bv:=(v_i)$. Fix $k,k_1,k_2,k_3$ and denote by 
\begin{eqnarray*}
a_{mn}&=&e^{2\pi bp_i^{k_m}}\cdot F_i^{k_n}(\bv)-F_i^{k_n}(\bv),\\
b_{n}&=&e^{2\pi bp_j^{k}}\cdot F_i^{k_n}(\bv)-F_i^{k_n}(\bv),\\
c_{n}&=&e^{2\pi bp_i^{k_n}}\cdot F_j^{k}(\bv)-F_j^{k}(\bv).
\end{eqnarray*}
That is, for each fixed terms we look at how much the factor of $F_j$ is shifted by the action of the different components of $F_i$. The quantum Serre relations for simply-laced roots are proved in exactly the same way in \cite{FI}. However in the doubly-laced case the calculation is more involved. We observe from the explicit expressions that
\Eq{a_{mn}=2-a_{nm},\tab b_n=-2-2c_n,}
and furthermore $a_{mn}$ only takes value in $\{0,1,2\}$ while $b_n$ only takes values in $\{0,-2\}$. 
Then the quantum Serre relations for the long root $F_j$ is equivalent to the vanishing of the following expression
\begin{eqnarray*}
QSE &\iff& [2+b_1]_q\left([a_{12}+b_2]_q[a_{23}-a_{31}+b_3]_q+[2-a_{31}+b_3]_q[a_{12}-a_{23}+b_2]_q\right)\\
&&+[2+b_2]_q\left([a_{23}+b_3]_q[a_{31}-a_{12}+b_1]_q+[2-a_{12}+b_1]_q[a_{23}-a_{31}+b_3]_q\right)\\
&&+[2+b_3]_q\left([a_{31}+b_1]_q[a_{12}-a_{23}+b_2]_q+[2-a_{23}+b_2]_q[a_{31}-a_{12}+b_1]_q\right)\\
&=&0
\end{eqnarray*}
which can be checked directly. The quantum Serre relations for type $G_2$ can be checked directly using the explicit expression given in Section \ref{sec:G2}.
\end{proof}


\section{Positive representations of $\cU_q(\g_\R)$ of type $B_2=C_2$}\label{sec:B2}
\subsection{Lusztig's data and transformation}\label{sec:LdataB2}
From Section \ref{sec:Ldata}, we know that when $\g$ is of type $B_2=C_2$, the positive unipotent subgroup is parametrized by
\Eq{x_1(a)x_2(b)x_1(c)x_2(d)=x_2(d')x_1(c')x_2(b')x_1(a'),}
where $a,b,c,d,a',b',c',d'\in \R_{>0}$.
Let us choose the following root subgroup on $C_2=Sp(4,\R)$:
\Eq{\label{1tmatrix}x_s(t)=\veca{1&t&0&0\\0&1&0&0\\0&0&1&0\\0&0&-t&1},\tab x_l(t)=\veca{1&0&0&0\\0&1&0&0\\0&0&1&t\\0&0&0&1},}
where $x_s$ and $x_l$ correspond to the short and long root respectively.

Then we have the following transformation rules:
\begin{Lem}\cite[Thm 3.1]{BZ} We have
\begin{eqnarray}
a'=\frac{abc}{R},&&\tab b'=\frac{R^2}{S},\\
c'=\frac{S}{R},&&\tab d'=\frac{bc^2d}{S},
\end{eqnarray}
where
\Eq{R=ab+ad+cd,\tab S=a^2b+d(a+c)^2.}
Furthermore, the transformation 
\Eq{\label{phi}\phi:(a,b,c,d)\mapsto (a',b',c',d')} 
is an involution:
\Eq{(a'',b'',c'',d'')=(a,b,c,d).}
\end{Lem}

\subsection{Classical principal series representations}\label{sec:classicalB2}
Using the above transformation, we can find the classical principal series representations for $B_2$ by commuting the corresponding root subgroup to the front. Under the Mellin transform, let $t, v$ be the variables corresponding to the Lusztig's parameters $a, c$ of the short root, while $u, w$ correspond to the parameters $b, d$ of the long root.

\begin{Prop} Corresponding to $w_0=s_1s_2s_1s_2$ where $1$ is short and $2$ is long, we have 
\begin{eqnarray}
E_1&=&\frac{d}{b}\del[,a]+\frac{2d}{c}\del[,b]+(1-\frac{d}{b})\del[,c]-\frac{2d}{c}\del[,d],\\
E_2&=&\del[,d],
\end{eqnarray}
so that the corresponding Mellin transformed action on $f(t,u,v,w)$ is given by
\begin{eqnarray}
E_1:f&\mapsto&(1+t)f(t+1,u+1,w-1),\\
&&+(1+2u-v)f(u+1,v+1,w-1)+(1+v-2w)f(v+1)\nonumber\\
E_2:f&\mapsto&(1+w)f(w+1).
\end{eqnarray}
(For notational convenience, the unshifted variables are omitted in the argument.)

On the other hand, corresponding to $w_0=s_2s_1s_2s_1$ we have
\begin{eqnarray}
E_1&=&\del[,a],\\
E_2&=&-\frac{a}{b}(1+\frac{a}{c})\del[,a]+(1-\frac{a^2}{c^2})\del[,b]+\frac{a}{b}(\frac{a}{c}+1)\del[,c]+\frac{a^2}{c^2}\del[,d],
\end{eqnarray}
so that the corresponding Mellin transformed action on $f(t,u,v,w)$ is given by
\begin{eqnarray}
E_1:f&\mapsto&(1+t)f(t+1),\\
E_2:f&\mapsto&(1-t+u)f(u+1)+(2-t+v)f(a-1,u+1,v+1)\\
&&+(1-u+v)f(t-2,u+1,v+2)+(1+w)f(t-2,v+2,w+1).\nonumber
\end{eqnarray}
The actions of $F_i$ and $H_i$ are given in Definition \ref{FH}.
\end{Prop}
\begin{proof} The action of the differential operators follows from the differentiation of the group action \eqref{groupaction} by the matrices given explicitly in \eqref{1tmatrix}, and then the Mellin transformed action follows from definition \eqref{MellinTransform}.
\end{proof}

\subsection{Explicit expressions}\label{sec:explicit}
Following the work in \cite{Ip}, we will construct the positive representations for $\cU_q(\g_\R)$ by quantizing the weights of the Mellin transformed action appropriately, and introduce certain twisting to make the actions positive. 

Recall from Definition \ref{qi} that $q_1=q^\frac{1}{2}=e^{\pi i b_s^2}$ and $q_2=q=e^{\pi ib^2}$. Also let $Q_s=b_s+b_s\inv$. 

\begin{Thm}\label{B2} The positive representations corresponding to $w_0=s_1s_2s_1s_2$  for $\cU_q(\g_\R)$ where $\g$ is of type $B_2$, acting on $L^2(\R^4)$, is given by

\begin{eqnarray} 
E_1&=&\left[\frac{Q_s}{2b_s}-\frac{i}{b}t\right]_{q_1}e^{2\pi b(-p_t-p_u+p_w)}+\left[\frac{Q_s}{2b_s}-\frac{i}{b}(2u-v)\right]_{q_1}e^{2\pi b(-p_u-p_v+p_w)}\nonumber\\
&&+\left[\frac{Q_s}{2b_s}-\frac{i}{b}(v-2w)\right]_{q_1}e^{-2\pi bp_v}\\
E_2&=&\left[\frac{Q}{2b}-\frac{i}{b}w\right]_{q_2}e^{-2\pi bp_w}\\
F_1&=&\left[\frac{Q_s}{2b_s}+\frac{i}{b}(2\l_1+t)\right]_{q_1}e^{2\pi bp_t}+\left[\frac{Q_s}{2b_s}+\frac{i}{b}(2\l_1+2t-2u+v)\right]_{q_1}e^{2\pi bp_v}\\
F_2&=&\left[\frac{Q}{2b}+\frac{i}{b}(2\l_2-t+u)\right]_{q_2}e^{2\pi bp_u}+\left[\frac{Q}{2b}+\frac{i}{b}(2\l_2-t+2u-v+w)\right]_{q_2}e^{2\pi bp_w}\nonumber\\\\
K_1&=&q_1^{H_1}=e^{\pi b(-\l_1-t+u-v+w)}\\
K_2&=&q_2^{H_2}=e^{\pi b(-2\l_2+t-2u+v-2w)}.
\end{eqnarray}
Note that for $q=e^{\pi ib^2}$,
\Eq{\left[\frac{Q}{2b}-\frac{i}{b}u\right]_{q}e^{2\pi bp}=e^{\pi b(u+2p)}+e^{\pi b(-u+2p)}}
is a positive essentially self-adjoint operator whenever $[p,u]=\frac{1}{2\pi i}$.

On the other hand,  the positive representations corresponding to $w_0=s_2s_1s_2s_1$ is more complicated, and is given by
\begin{eqnarray}
E_1&=&\left[\frac{Q_s}{2b_s}-\frac{i}{b}t\right]_{q_1}e^{-2\pi bp_t}\\
E_2&=&\left[\frac{Q}{2b}-\frac{i}{b}w\right]_{q_2}e^{2\pi b(2p_t-2p_v-p_w)}+\left[\frac{Q}{2b}-\frac{i}{b}(u-t)\right]_{q_2}e^{-2\pi bp_u}\\
&&+[2]_{q_1}\left[\frac{Q}{2b}-\frac{i}{2b}(v-t)\right]_{q_2}e^{2\pi b(p_t-p_u-p_v)}+\left[\frac{Q}{2b}-\frac{i}{b}(v-u)\right]_{q_2}e^{2\pi b(2p_t-p_u-2p_v)}\nonumber\\
F_1&=&\left[\frac{Q_s}{2b_s}+\frac{i}{b}(2\l_1+v-2w)\right]_{q_1}e^{2\pi bp_v}+\left[\frac{Q_s}{2b_s}+\frac{i}{b}(2\l_1+t-2u+2v-2w)\right]_{q_1}e^{2\pi bp_t}\nonumber\\\\
F_2&=&\left[\frac{Q}{2b}+\frac{i}{b}(2\l_2+w)\right]_{q_2}e^{2\pi bp_w}+\left[\frac{Q}{2b}+\frac{i}{b}(2\l_2+u-v+2w)\right]_{q_2}e^{2\pi bp_u}\\
K_1&=&e^{\pi b(-\l_1-t+u-v+w)}\\
K_2&=&e^{\pi b(-2\l_2+t-2u+v-2w)}.
\end{eqnarray}
\end{Thm}

Hence one can see that the expression resembles that of the classical formula. However, it turns out that it is more natural to consider the rescaled version, where for variables corresponding to the short root, we rescale by $\sqrt{2}$, so that
\begin{eqnarray*}
bu&\mapsto& b\sqrt{2}u=2b_su, \\
bp_u&\mapsto& \frac{b}{\sqrt{2}}p_u=b_sp_u.
\end{eqnarray*}
We will also rescale the parameters $\l_s$ corresponding to short roots to $\sqrt{2}\l_s$.

Hence let us introduce the following notation. Let $u_s$ and $u_l$ be a linear combinations of Mellin transformed variables corresponding to the short root and long root respectively. (This also applies to the parameters $\l_i$). Also let $p_s$ and $p_l$ be the corresponding shifting operators. 

\begin{Def}\label{notation} We denote by 
\Eq{[u_s+u_l]e(p_s+p_l) =e^{\pi (b_s u_s+b_lu_l)+2\pi (b_sp_s+ b_lp_l)}+e^{-\pi (b_s u_s+b_lu_l)+2\pi (b_sp_s+ b_lp_l)}.}
\end{Def}

Then under the above rescaling, we can rewrite Theorem \ref{B2} as follows:
\begin{MThm}[\ref{B2}*]
Let 
\begin{eqnarray}
e_i&:=&2\sin(\pi b_i^2)E_i=\left(\frac{i}{q_i-q_i\inv}\right)\inv E_i,\\
f_i&:=&2\sin(\pi b_i^2)F_i=\left(\frac{i}{q_i-q_i\inv}\right)\inv F_i.
\end{eqnarray}
Then the positive representations corresponding to $w_0=s_1s_2s_1s_2$  for $\cU_q(\g_\R)$ where $\g$ is of type $B_2$, acting on $L^2(\R^4)$, is given by

\begin{eqnarray} 
e_1&=&\label{e1B}[t]e(-p_t-p_u+p_w)+[u-v]e(-p_u-p_v+p_w)+[v-w]e(-p_v)\\
e_2&=&\label{e2B}[w]e(-p_w)\\
f_1&=&[-2\l_1-t]e(p_t)+[-2\l_1-2t+u-v]e(p_v)\\
f_2&=&[-2\l_2+2t-u]e(p_u)+[-2\l_2+2t-2u+2v-w]e(p_w)\\
K_1&=&e^{\pi b_s(-2\l_1-2t-2v)}e^{\pi b(u+w)}\\
K_2&=&e^{\pi b(-2\l_2-2u-2w)}e^{\pi b_s(2t+2v)}.
\end{eqnarray}

On the other hand,  the positive representations corresponding to $w_0=s_2s_1s_2s_1$ is given by
\begin{eqnarray}
e_1&=&\label{e1C}[t]e(-p_t)\\
e_2&=&\label{e2C} [w]e(2p_t-2p_v-p_w)+[u-2t]e(-p_u)\\
&&+[2]_{q_1}[v-t]e(p_t-p_u-p_v)+[2v-u]e(2p_t-p_u-2p_v)\nonumber\\
f_1&=&[-2\l_1-v+w]e(p_v)+[-2\l_1-t+u-2v+w]e(p_t)\\
f_2&=&[-2\l_2-w]e(p_w)+[-2\l_2-u+2v-2w]e(p_u)\\
K_1&=&e^{\pi b_s(-2\l_1-2t-2v)}e^{\pi b(u+w)}\\
K_2&=&e^{\pi b(-2\l_2-2u-2w)}e^{\pi b_s(2t+2v)}.
\end{eqnarray}
\end{MThm} 

\begin{proof} The commutation relations of the operators can be checked directly. Note that the action of $F_i$ and $K_i$ coincides with the one given in Definition \ref{FH}.
\end{proof}

\subsection{Transformations of operators}\label{sec:trans}
As in the simply-laced case, by quantizing $\phi$ from \eqref{phi}, there is a unitary transformation $\Phi$ that intertwines the above action corresponding to the change of reduced expression $w_0=s_1s_2s_1s_2=s_2s_1s_2s_1$.

\begin{Thm}\label{trans2} 
We define the transformation 
\begin{eqnarray*}
\Phi: L^2(\R^4)&\to& L^2(\R^4)\\
f(t,u,v,w)&\mapsto& \Phi f:=F(t',u',v',w')
\end{eqnarray*}
by
\Eq{\Phi:=T\circ \Phi_3\circ \Phi_2\circ \Phi_1,}
where
\begin{eqnarray}
\Phi_1&=&\frac{g_b(e^{\pi b(u+w)-2\pi b_sv+2\pi b(p_w-p_u)})}{g_b(e^{-\pi b(u+w)+2\pi b_sv+2\pi b(p_w-p_u)})},\\
\Phi_2&=&\frac{g_{b_s}(e^{\pi b_s(t-v)+\pi bw+2\pi b_s(p_v-p_t)+2\pi b(p_w-p_u)})}{g_{b_s}(e^{-\pi b_s(t-v)-\pi bw+2\pi b_s(p_v-p_t)+2\pi b(p_w-p_u)})},\\
\Phi_3&=&\frac{g_b(e^{2\pi b_st+\pi b(w-u)+2\pi b_s(2p_v-2p_t)+2\pi b(p_w-p_u)})}{g_b(e^{-2\pi b_st-\pi b(w-u)+2\pi b_s(2p_v-2p_t)+2\pi b(p_w-p_u)})},
\end{eqnarray}
and $T$ is the transformation matrix of determinant $-1$:
\Eq{\veca{t'\\u'\\v'\\w'}:=T \veca{t\\u\\v\\w}=\veca{0&0&1&-\sqrt{2}\\0&1&0&0\\1&0&0&\sqrt{2}\\0&0&0&1}\veca{t\\u\\v\\w}.}
Then $\Phi$ is a unitary transformations, $\Phi^2=1$, and for any operators $X$,
\Eq{X\longmapsto \Phi \circ X \circ \Phi\inv}
gives the corresponding action on $s_1s_2s_1s_2\corr s_2s_1s_2s_1$.
\end{Thm}
\begin{proof} The ratios of $g_b$ is well-defined since the argument commute with each other. Furthermore they are unitary by Lemma \ref{unitary}. For the proof, it suffices to apply the conjugation properties of $g_b$ from Lemma \ref{qsum} to the positive representation of type $B_2$ given by Theorem \ref{B2}. Note that the factor $[2]_{q_1}$ from $E_2$ is obtained by applying $\Phi_2$ using \eqref{qsum3}-\eqref{qsum4}.
\end{proof}


\section{Positive representations of $\cU_q(\g_\R)$}\label{sec:posrep}
From Theorem \ref{trans2}, together with the result of \cite[Theorem 5.2]{Ip}, we can find explicit expressions of the positive representation starting from any reduced expression of $w_0$, followed by applying the unitary transformations to the desired reduced expression. Consequently, all the operators will be positive essentially self-adjoint. Furthermore, by bringing $s_is_js_is_j$ to the front of $w_0$, all the commutation relations between $E_i$ and $E_j$ will be easy to check, while the commutation relations involving $F_i$ follows from Theorem \ref{Serre} and their explicit expressions.

Actually, we only need the following transformations rules:

\begin{Prop}\label{rules}
If $i,j$ are not connected in the Dynkin diagram, corresponding to $...s_is_j...\corr ...s_js_i...$, the operators transformed on $f(u,v)$ simply by 
\Eq{u\corr v.}

For $i,j$ connected by a single edge in the Dynkin diagram, corresponding to \\$...s_is_js_i...\corr ...s_js_is_j...$, the operators given by \cite[Theorem 5.2]{Ip} transformed on $f(u,v,w)$ as
\Eq{[w]e(-p_w)\corr [u]e(-p_u-p_v+p_w)+[v-w]e(-p_v).}

For $i,j$ connected by a double edge in the Dynkin diagram, corresponding to\\$...s_is_js_is_j...\corr ...s_js_is_js_i...$ the operators given by Theorem \ref{trans2} transformed on $f(t,u,v,w)$ as
\Eq{\eqref{e1B}\corr\eqref{e1C},\tab \eqref{e2B}\corr \eqref{e2C}.}
\end{Prop}
 
Let us also rewrite the action of $F_i$ and $K_i$ from Definition \ref{FH} in terms of the rescaled variables using Definition \ref{notation}.

\begin{MDef}[\ref{FH}*]Let $$f_i=2\sin\pi b_i^2 F_i=\left(\frac{i}{q_i-q_i\inv}\right)\inv F_i.$$
Then the quantized action of $F_i$ is given by
\begin{eqnarray}
f_i&=&\sum_{k=1}^n[-a_{r(j),i}v_j+u_i^k-2\l_i]e(p_i^k)\\
K_i&=&e^{-2\pi b_i\l_i-\sum_{j=1}^{l(w_0)}\pi b_ja_{r(j),i}v_j}
\end{eqnarray}
where $v_j$ is the labeling given in Definition \ref{labelling}.
\end{MDef}

\subsection{Type $B_n$} \label{sec:Bn}
Let us choose the following reduced expression for $w_0$
$$w_0=1212\;\;\;32123\;\;\;4321234\;\;\; ... n(n-1)...1...n$$
where for simplicity, we denote by $i:=s_i$. Then by transposing the desired index to the right, and applying the rules from Proposition \ref{rules} repeatedly, we obtain
\begin{Prop} The action of $E_1$ is given by
\Eq{e_1=\sum_{k=1}^n \left[u_1^k-u_2^{2k-1}\right]e\left(-p_1^k-\sum_{l=1}^{2k-2}(-1)^lp_2^l\right)+\sum_{k=1}^{n-1}\left[u_2^{2k}-u_1^k\right]e\left(-p_1^k-\sum_{l=1}^{2k}(-1)^lp_2^l\right).}
Note that the variable $u_2^{2n-1}=0$ is non-existent.

The action of $E_i$ for $i\geq 2$ is given by
\Eq{\label{BnEi}e_i=\sum_{k=1}^{2(n-i)+1}\left[(-1)^k(u_{i+1}^k-u_i^k)\right]e\left(\sum_{l_1=1}^{s_1(k)}(-1)^{l_1}p_i^{l_1}+\sum_{l_2=1}^{s_2(k)}(-1)^{l_2}p_{i+1}^{l_2}\right),}
where $e_i=2\sin\pi b_i^2 E_i$ and
\begin{eqnarray*}
 s_1(k)&:=&2\left\lceil\frac{k}{2}\right\rceil  -1,\\
s_2(k)&:=&2\left\lfloor\frac{k}{2}\right\rfloor.
\end{eqnarray*}

\end{Prop}

\subsection{Type $C_n$}\label{sec:Cn}
Using the exact same expression for $w_0$ as in type $B_n$, we have the following action.

\begin{Prop} The action of $E_1$ is given by
\begin{eqnarray}
e_1&=&\sum_{k=1}^n \left[u_1^k-2u_2^{2k-1}\right]e\left(-p_1^k-2\sum_{l=1}^{2k-2}(-1)^lp_2^l\right)\nonumber\\
&&+[2]_{q_s}\sum_{k=1}^{n-1}[u_2^{2k}-u_2^{2k-1}]e\left(-p_1^k-2\sum_{l=1}^{2k-2}(-1)^lp_2^l+p_2^{2k-1}-p_2^{2k}\right)\nonumber\\
&&+\sum_{k=1}^{n-1}\left[2u_2^{2k}-u_1^k\right]e\left(-p_1^k-2\sum_{l=1}^{2k}(-1)^lp_2^l\right),
\end{eqnarray}
while the action of $E_i$ for $i\geq 2$ is the same as \eqref{BnEi}.
\end{Prop}

\subsection{Type $F_4$} \label{sec:F4}
Let us choose the following reduced expression for $w_0$:
$$w_0=3232\;\;\;12321\;\;\;432312343213234$$
which follows from the embedding $B_2\subset B_3\subset F_4$. We will use the notation introduced in \cite{Ip2} to simplify the expressions. Let
\begin{eqnarray*}
P_1&=&-p_1^3+p_1^2-p_2^4+2p_3^4-2p_3^5+2p_3^2-2p_3^3+p_2^1-p_1^1\\
P_2&=&-p_2^7+p_2^6-p_1^4+p_1^3-p_2^5+2p_3^5-2p_3^6+p_2^3-p_1^2+p_1^1-p_2^2+2p_3^1-2p_3^2\\
P_3&=&-p_3^9-p_2^8+p_2^7-p_2^6+p_2^5+p_3^6-p_4^3-p_3^5+p_3^3-p_2^3+p_2^2+p_3^2-p_3^4+p_4^1-p_3^1
\end{eqnarray*}
and let $P_i(x)$ be the partial sum of $P_i$ from $x$ (ignoring the coefficient) to the right most term. For example
$$P_1(p_3^2) := 2p_3^2-2p_3^3+p_2^1-p_1^1.$$

\begin{Prop} The action of $E_i$ is given as follows:
{\small
\begin{eqnarray}
e_1&=&[u_1^3]e(P_1)+[u_2^4-u_1^2]e(P_1(p_2^4))+[u_2^3-2u_3^4]e(-p_2^3+P_1(p_3^2))\nonumber\\
&&+[2]_{q_s}[u_3^5-u_3^4]e(-p_2^3+p_3^4-p_3^5+P_1(p_3^2))+[2u_3^5-u_2^3]e(-p_2^3+P_1(p_3^4))\nonumber\\
&&+[u_2^2-2u_3^2]e(-p_2^2+p_2^1-p_1^1)+[2]_{q_s}[u_3^3-u_3^2]e(-p_2^2+p_3^2-p_3^3+p_2^1-p_1^1)\nonumber\\
&&+[2u_3^3-u_2^2]e(-p_2^2+P_1(p_3^2))+[u_1^1-u_2^1]e(-p_1^1),\\
e_2&=&[u_2^7]e(P_2)+[u_1^4-u_2^6]e(P_2(p_1^4))+[u_2^5-u_1^3]e(P_2(p_2^5))+[u_2^4-2u_3^5]e(p_2^3-p_2^4+P_2(p_1^2))\nonumber\\
&&+[2]_{q_s}[u_3^6-u_3^5]e(p_3^5-p_3^6-p_2^4+P_2(p_2^3))+[2u_3^6-u_2^4]e(-p_2^4+P_2(p_3^5))\nonumber\\
&&+[u_1^2-u_2^3]e(P_2(p_1^2))+[u_2^2-u_1^1]e(P_2(p_2^2))+[2u_3^2-u_2^1]e(-p_2^1+P_2(p_3^1))\nonumber\\
&&+[2]_{q_s}[u_3^2-u_3^1]e(p_3^1-p_3^2-p_2^1)+[u_2^1-2u_3^1]e(-p_2^1),\\
e_3&=&[u_3^9]e(P_3)+[u_2^8-u_3^8]e(-p_3^8+P_3(p_2^8))+[u_3^8-u_2^7]e(-p_3^8+P(p_2^6))+[u_2^6-u_3^7]e(-p_3^7+P(p_2^6))\nonumber\\
&&+[u_3^7-u_2^5]e(-p_3^7+P_3(p_3^6))+[u_4^3-u_3^6]e(P_3(p_4^3))+[u_3^5+u_3^4-u_4^2]e(-p_4^2+P(p_3^5))\nonumber\\
&&+[u_3^5-u_3^3]e(-p_4^2-p_3^5+p_3^4+P_3(p_2^3))+[u_2^3-u_3^4-u_3^3]e(-p_4^2+P_3(p_2^3))\nonumber\\
&&+[u_3^4+u_3^3-u_2^2]e(-p_4^2+P_3(p_3^2))+[u_3^4-u_3^2]e(-p_4^2+p_3^3+P_3(p_3^4))\nonumber\\
&&+[u_4^2-u_3^3-u_3^2]e(-p_4^2+p_4^1-p_3^1)+[u_3^1-u_4^1]e(-p_3^1),\\
e_4&=&[u_4^1]e(-p_4^1).
\end{eqnarray}}
\end{Prop}

\section{Positive representations of $\cU_q(\g_\R)$ for type $G_2$}\label{sec:G2}
Let $w_0=s_2s_1s_2s_1s_2s_1$ where 1 is the long root and 2 is the short root. First let us descirbe the root subgroup $x_1(t)$ and $x_2(t)$ embedded in $SL(7)$, which can be found in \cite{HR}.
\begin{eqnarray}
x_1(t)&=&\veca{1&0&0&0&0&0&0\\0&1&t&0&0&0&0\\0&0&1&0&0&0&0\\0&0&0&1&0&0&0\\0&0&0&0&1&t&0\\0&0&0&0&0&1&0\\0&0&0&0&0&0&1}\\
x_2(t)&=&\veca{1&-t&0&0&0&0&0\\0&1&0&0&0&0&0\\0&0&1&-t&-t^2&0&0\\0&0&0&1&2t&0&0\\0&0&0&0&1&0&0\\0&0&0&0&0&1&t\\0&0&0&0&0&0&1}
\end{eqnarray}
Solving for the matrix coefficients explicitly by Mathematica, we found the relations between the Lusztig's parametrization:
\Eq{x_1(a)x_2(b)x_1(c)x_2(d)x_1(e)x_2(f)=x_2(a')x_1(b')x_2(c')x_1(d')x_2(e')x_1(f')}
(an explicit relation can also be found in \cite[Thm 3.1]{BZ}). From this we can derive the classical principal series representation:
\begin{Prop} The classical principal series representation corresponding to $w_0=s_2s_1s_2s_1s_2s_1$ is given by
\begin{eqnarray}
E_1&=&\del[,f],\\
E_2&=&\frac{ef}{bc}\del[,a]+\frac{3ef}{c^2}\del[,b]+\left(\frac{2f}{d}-\frac{ef}{bc}+\frac{2ef}{cd}\right)\del[,c]+\left(\frac{3f}{e}-\frac{3ef}{c^2}\right)\del[,d]\nonumber\\
&&+\left(1-\frac{2f}{d}-\frac{2ef}{cd}\right)\del[,e]-\frac{3f}{e}\del[,f],
\end{eqnarray}
and the Mellin transformed action on $f(r,s,t,u,v,w)$ is given by
\begin{eqnarray}
E_1:f&\mapsto& (w+1)f(w+1), \\
E_2:f&\mapsto&(1+r)f(r+1,s+1,t+1,v-1,w-1)\nonumber\\
&&+(1+3s-t)f(s+1,t+2,v-1,w-1)\nonumber\\
&&+(1+3u-2v)f(u+1,v+1,w-1)\nonumber\\
&&+(2+2t-2v)f(t+1,u+1,w-1)\nonumber\\
&&+(1+2t-3u)f(t+2,u+1,v-1,w-1)\nonumber\\
&&+(1+v-3w)f(v+1).
\end{eqnarray}
Again, for notational convenience, the unshifted variables are omitted in the argument.
\end{Prop}
To get the quantized action, we again rescaled the variables corresponding to whether the index are long ($s,u,w$) or short  ($r,t,v$). Using the notation from Definition \ref{notation} (with $b_s=\frac{b}{\sqrt{3}}$), we found the action as follows.
\begin{Thm}\label{G2} The action of $\cU_q(\g_\R)$ on $L^2(\R^6)$, where $\g$ is of type $G_2$ corresponding to \\${w_0=s_2s_1s_2s_1s_2s_1}$  is given by
\begin{eqnarray}
e_1&=&[w]e(-p_w),\\
e_2&=&[r]e(-p_r-p_s-p_t+p_v+p_w)+[s-t]e(-p_s-2p_t+p_v+p_w)\nonumber\\
&&+[u-2v]e(-p_u-p_v+p_w)+[2]_{q_s}[t-v]e(-p_t-p_u+p_w)\nonumber\\
&&+[2t-u]e(-2p_t-p_u+p_v+p_w)+[v-w]e(-p_v),
\end{eqnarray}
where again $e_i=2\sin\pi b_i^2 E_i$. The action of $F_i$ and $K_i$ are given by Definition \ref{FH}* as before.
\end{Thm}

The action corresponding to $w_0=s_1s_2s_1s_2s_1s_2$ can also be computed, and $E_1$ consists of 13 terms. It is unitary equivalent to the above representation via a transformation consisting of 11 quantum dilogarithms with parameters $b$ or $b_s$, and a linear change of variables $T$ with $\det(T)=1$, generalizing Theorem \ref{trans2} . It can instead be obtained from the folding of the positive representations of type $D_4$ described in the next section.

\section{Folding of representations}\label{sec:folding}
By comparing the action of $C_n$ given in Section \ref{sec:Cn} and the action of $D_{n+1}$ given in \cite{Ip}, there is a strong similarity between the action. In fact the action of $E_i$ for $i\geq 2$ are identical. This holds in general due to the principle of folding of Dynkin diagram. 

It is known that the philosophy of folding in the setting of quantum groups is more complicated \cite{BG}. Nevertheless, with the mixture of classical construction, Mellin transformation and quantization described in this paper, we can still obtain the description of the positive representations in the non-simply-laced case from the corresponding unfolded simply-laced type by means of the folding of pinning \cite[1.5]{Lu2}. Explicitly:

\begin{Thm}Let us use the labeling introduced in Section \ref{sec:Uq} and \cite{Ip}.

The positive representations of $C_n$ corresponding to $w_0$ can be obtained from the positive representations of $D_{n+1}$ with $s_1$ in $w_0(C_n)$ replaced by $s_0s_1$ in $w_0(D_{n+1})$.

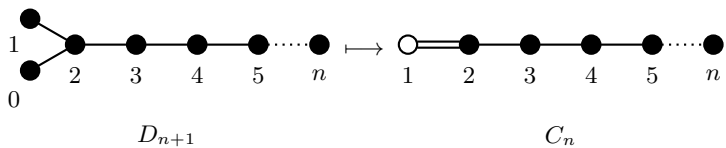
\begin{center}
  \begin{tikzpicture}[scale=.4,baseline=-1ex]
    \draw[thick,fill=black] (150: 17 mm) circle (.3cm);
    \draw[thick,fill=black] (-150: 17 mm) circle (.3cm);
    \foreach \x in {0,...,4}
    \draw[xshift=\x cm,thick,fill=black] (\x cm,0) circle (.3cm);
    \draw[thick] (150: 3 mm) -- (150: 14 mm);
    \draw[thick] (-150: 3 mm) -- (-150: 14 mm);
    \foreach \y in {0.15,...,2.15}
    \draw[xshift=\y cm,thick] (\y cm,0) -- +(1.4 cm,0);
    \draw[dotted,thick] (6.3 cm,0) -- +(1.4 cm,0);    
    \foreach \z in {2,...,5}
    \node at (2*\z-4,-1) {$\z$};
    \node at (8,-1){$n$};
   \node at (-2,-1.8){$0$};
  \node at (-2, 0){$1$};
  \node at (3, -3){$D_{n+1}$};
  \end{tikzpicture}
$\longmapsto$
  \begin{tikzpicture}[scale=.4,baseline=-1ex]
    \draw[xshift=0 cm,thick] (0 cm, 0) circle (.3 cm);
    \foreach \x in {1,...,5}
    \draw[xshift=\x cm,thick,fill=black] (\x cm,0) circle (.3cm);
    \draw[dotted,thick] (8.3 cm,0) -- +(1.4 cm,0);
    \foreach \y in {1.15,...,3.15}
    \draw[xshift=\y cm,thick] (\y cm,0) -- +(1.4 cm,0);
    \draw[thick] (0.3 cm, .1 cm) -- +(1.4 cm,0);
    \draw[thick] (0.3 cm, -.1 cm) -- +(1.4 cm,0);
    \foreach \z in {1,...,5}
    \node at (2*\z-2,-1) {$\z$};
\node at (10,-1){$n$};
\node at (5, -3){$C_n$};
  \end{tikzpicture}
\end{center}

The positive representations of $B_n$ corresponding to $w_0$ can be obtained from the positive representations of $A_{2n-1}$  with $s_{n-k+1}$ in $w_0(B_n)$ replaced by $s_ks_{2n-k}$ in $w_0(A_{2n-1})$ for $1\leq k\leq n-1$.

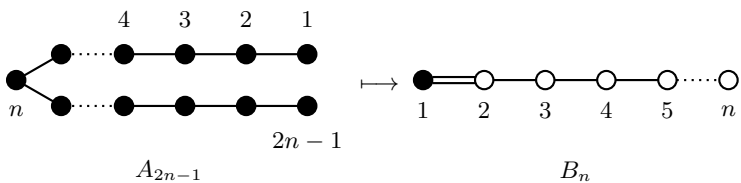
\begin{center}
  \begin{tikzpicture}[scale=.4,baseline=-1ex]
\draw[thick,fill=black] (0,0) circle(.3cm);
    \draw[thick,fill=black] (30: 17 mm) circle (.3cm);
    \draw[thick,fill=black] (-30: 17 mm) circle (.3cm);
    \foreach \x in {2,...,5}{
    \draw[xshift=\x cm -0.47cm,thick,fill=black] (\x cm,0.85) circle (.3cm);
     \draw[xshift=\x cm-0.47cm,thick,fill=black] (\x cm,-0.85) circle (.3cm);}
    \draw[thick] (30: 3 mm) -- (30: 14 mm);
    \draw[thick] (-30: 3 mm) -- (-30: 14 mm);
    \foreach \y in {2.15,...,4.15}{
    \draw[xshift=\y cm -0.47cm,thick] (\y cm,0.85) -- +(1.4 cm,0);
    \draw[xshift=\y cm -0.47cm,thick] (\y cm,-0.85) -- +(1.4 cm,0);}
    \draw[dotted,thick] (1.83 cm,0.85) -- +(1.4 cm,0);    
    \draw[dotted,thick] (1.83 cm,-0.85) -- +(1.4 cm,0);  
    \node at (9.53,2) {$1$};
   \node at (7.53,2) {$2$};
   \node at (5.53,2) {$3$};
   \node at (3.53,2) {$4$};
    \node at (0,-1) {$n$};
\node at (9.53,-2) {$2n-1$};
\node at (5, -3){$A_{2n-1}$};
  \end{tikzpicture}
$\longmapsto$
  \begin{tikzpicture}[scale=.4,baseline=-1ex]
    \draw[xshift=0 cm,thick,fill=black] (0 cm, 0) circle (.3 cm);
    \foreach \x in {1,...,5}
    \draw[xshift=\x cm,thick] (\x cm,0) circle (.3cm);
    \draw[dotted,thick] (8.3 cm,0) -- +(1.4 cm,0);
    \foreach \y in {1.15,...,3.15}
    \draw[xshift=\y cm,thick] (\y cm,0) -- +(1.4 cm,0);
    \draw[thick] (0.3 cm, .1 cm) -- +(1.4 cm,0);
    \draw[thick] (0.3 cm, -.1 cm) -- +(1.4 cm,0);
    \foreach \z in {1,...,5}
    \node at (2*\z-2,-1) {$\z$};
\node at (10,-1){$n$};
\node at (5, -3){$B_n$};
  \end{tikzpicture}
\end{center}

The positive representations of $F_4$ corresponding to $w_0$ can be obtained from the positive representations of $E_6$ with $s_1,s_2,s_3,s_4$ in $w_0(F_4)$ replaced by $s_1s_5,s_2s_4,s_3,s_0$ in $w_0(E_6)$.

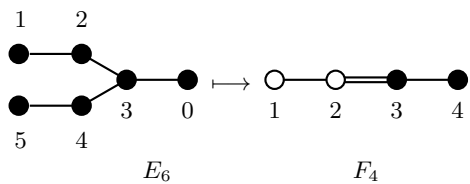
\begin{center}  
\begin{tikzpicture}[scale=.4,baseline=-1ex]
\draw[thick,fill=black] (2,0) circle(.3cm);
\draw[thick,fill=black] (0,0) circle(.3cm);
    \draw[thick,fill=black] (150: 17 mm) circle (.3cm);
    \draw[thick,fill=black] (-150: 17 mm) circle (.3cm);
    \draw[xshift=-2 cm +0.47cm,thick,fill=black] (-2 cm,0.85) circle (.3cm);
     \draw[xshift=-2 cm+0.47cm,thick,fill=black] (-2 cm,-0.85) circle (.3cm);
    \draw[xshift=-0.85 cm,thick] (1.15 cm,0) -- +(1.4 cm,0);
    \draw[thick] (150: 3 mm) -- (150: 14 mm);
    \draw[thick] (-150: 3 mm) -- (-150: 14 mm);
    \draw[xshift=1.15 cm -0.47cm,thick] (-3.85 cm,0.85) -- +(1.4 cm,0);
    \draw[xshift=1.15 cm -0.47cm,thick] (-3.85  cm,-0.85) -- +(1.4 cm,0);
   \node at (-3.47,2) {$1$};
   \node at (-1.47,2) {$2$};
    \node at (0,-1) {$3$};
   \node at (-1.47,-2) {$4$};
   \node at (-3.47,-2) {$5$};
\node at (2,-1) {$0$};
\node at (1, -3){$E_6$};
  \end{tikzpicture}
$\longmapsto$
  \begin{tikzpicture}[scale=.4,baseline=-1ex]
    \draw[thick] (-2 cm ,0) circle (.3 cm);
	\node at (-2,-1) {$1$};
    \draw[thick] (0 ,0) circle (.3 cm);
	\node at (0,-1) {$2$};
    \draw[thick,fill=black] (2 cm,0) circle (.3 cm);
	\node at (2,-1) {$3$};
    \draw[thick,fill=black] (4 cm,0) circle (.3 cm);
	\node at (4,-1) {$4$};
    \draw[thick] (15: 3mm) -- +(1.5 cm, 0);
    \draw[xshift=-2 cm,thick] (0: 3 mm) -- +(1.4 cm, 0);
    \draw[thick] (-15: 3 mm) -- +(1.5 cm, 0);
    \draw[xshift=2 cm,thick] (0: 3 mm) -- +(1.4 cm, 0);
\node at (1, -3){$F_4$};
  \end{tikzpicture}
\end{center}

Finally, the positive representations of $G_2$ can be obtained from the positive representations of $D_4$ with $s_1$ in $w_0(G_2)$ replaced by $s_0s_1s_3$ in $w_0(D_4)$.

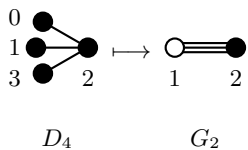
\begin{center}  
\begin{tikzpicture}[scale=.4,baseline=-1ex]
\draw[thick,fill=black] (0,0) circle(.3cm);
    \draw[thick,fill=black] (150: 17 mm) circle (.3cm);
    \draw[thick,fill=black] (-180: 17 mm) circle (.3cm);
    \draw[thick,fill=black] (-150: 17 mm) circle (.3cm);
    \draw[xshift=-0.85 cm,thick] (0.85 cm,0) -- +(-1.4 cm,0);
    \draw[thick] (150: 3 mm) -- (150: 14 mm);
    \draw[thick] (-150: 3 mm) -- (-150: 14 mm);
   \node at (-2.4,1) {$0$};
    \node at (0,-1) {$2$};
   \node at (-2.4,-1) {$3$};
\node at (-2.4,0) {$1$};
\node at (-1, -3){$D_4$};
  \end{tikzpicture}
$\longmapsto$
  \begin{tikzpicture}[scale=.4,baseline=-1ex]
    \draw[thick] (0 ,0) circle (.3 cm);
    \draw[thick,fill=black] (2 cm,0) circle (.3 cm);
    \draw[thick] (30: 3mm) -- +(1.5 cm, 0);
    \draw[thick] (0: 3 mm) -- +(1.4 cm, 0);
    \draw[thick] (-30: 3 mm) -- +(1.5 cm, 0);
\node at (0,-1){$1$};
\node at (2,-1){$2$};
\node at (1, -3){$G_2$};
  \end{tikzpicture}
\end{center}

\end{Thm}
\begin{proof} First we rescaled $q$ so that for non-paired index $i$, the action of $E_i$ for $w_0=w_{l-1}s_i$ matches with the correct $q$. For example, we have to use
$\cU_{q_s}(D_{n+1})$ where $q_s=q^{\frac{1}{2}}$ so that it corresponds to the short roots in $\cU_q(C_n)$.

Next we identify the variables corresponding to the paired roots in the simply-laced case. For example we let $u_0^k=u_1^k$ and $p_0^k=p_1^k$ in $D_{n+1}$. These will simplify certain expressions, and occasionally produce identical terms.

Then we replace the parameter $b$ by the appropriate $b_i$ according to whether they correspond to short or long roots, much like the procedure described before Definition \ref{notation}.

Finally, for the identical terms, we quantize the multiples with $q_s$. For example
$$2[v-t]e(p_t-p_u-p_v)\longmapsto [2]_{q_s}[v-t]e(p_t-p_u-p_v).$$
In the case of type $G_2$, the representation corresponding to $w_0=s_1s_2s_1s_2s_1s_2$ has a factor 6 which is quantized to $[3]_{q_s}!$.
\end{proof}

\section{Transcendental relations}\label{sec:langlands}
We recall from the simply-laced case that, if we define the following operators $\til[e_i],\til[f_i],\til[K_i]$ as:
\begin{eqnarray}
\til[e_i]&=&(e_i)^{\frac{1}{b^2}}\\
\til[f_i]&=&(f_i)^{\frac{1}{b^2}}\\
\til[K_i]&=&(K_i)^{\frac{1}{b^2}}\\
\end{eqnarray}
Then the operators are precisely the same as replacing $b$ with $b\inv$, so that the operators
\begin{eqnarray*}
\til[E_i]&=&(2\sin\pi b^{-2})\inv e_i,\\
\til[F_i]&=&(2\sin\pi b^{-2})\inv f_i
\end{eqnarray*}
and $\til[K_i]$ generates a representation of $\cU_{\til[q]}(\g_\R)$ where $\til[q]=e^{\pi ib^{-2}}$. In particular, the transcendental relations means that
\begin{eqnarray}
\left([w]e(-p_w)\right)^{\frac{1}{b^2}}&=&\left(e^{\pi b(w-2p_w)}+e^{\pi b(-w-2p_w)}\right)^{\frac{1}{b^2}}\nonumber\\\label{E1}
&=&e^{\pi b\inv(w-2p_w)}+e^{\pi b\inv(-w-2p_w)}:=([w]e(-p_w))_*
\end{eqnarray}
which is due to Lemma \ref{qbi}.

However in the case of non-simply-laced type, the tilde generators no longer generates $\cU_{\til[q]}(\g_\R)$, but rather short roots become long roots and vice versa. More precisely, we have the following Theorem.

\begin{Thm}\label{main} Let $\til[q]=e^{\pi ib_s^{-2}}$ and define $\til[q]_i=e^{\pi ib_i^{-2}}$. Define the operators
\begin{eqnarray}
\til[e_i]&:=&(e_i)^{\frac{1}{b_i^2}}\\
\til[f_i]&:=&(f_i)^{\frac{1}{b_i^2}}\\
\til[K_i]&:=&(K_i)^{\frac{1}{b_i^2}}
\end{eqnarray}
where as before 
\begin{eqnarray}
e_i=2\sin\pi b_i^2E_i,&\tab& f_i=2\sin\pi b_i^2F_i,\\
\til[e_i]=2\sin\pi b_i^{-2}\til[E_i],&\tab&\til[f_i]=2\sin\pi b_i^{-2}\til[F_i].
\end{eqnarray}
Then the generators of $\cU_{\til[q]}({}^L\g_\R)$ are represented by the operators $\til[E_i],\til[F_i]$ and $\til[K_i]$, where ${}^L\g_\R$ is defined by replacing long roots with short roots and short roots with long roots in the Dynkin diagram of $\g$.
\end{Thm}

This is a nontrivial result, since other than exchanging $q_i$ with $\til[q_i]$, the corresponding quantum Serre relations also get interchanged, and there is no classical analogue of the above transcendental relations.

To give the proof of Theorem \ref{main}, let us introduce the following notation.

\begin{Def}Let us denote by \Eq{([u_s+u_l]e(p_s+p_l))_*} the same notation as in Definition \ref{notation}, however with $b$ and $b_s$ replaced by $b\inv$ and $b_s\inv$ respectively. 
\end{Def}

We begin with a lemma.

\begin{Lem}\label{B2trans} Let $E_i, F_i,K_i$ be the generators of $\cU_q(\g_\R)$ of type $B_2$ where $i=1$ corresponds to the short root, and $i=2$ corresponds to the long root. Then we have
\begin{eqnarray*}
q_1=e^{\pi ib_s^2}=q^{\frac{1}{2}},&\tab& q_2=e^{\pi ib^2}=q,\\
\til[q_1]=e^{\pi ib_s^{-2}}=\til[q],&\tab&\til[q_2]=e^{\pi ib^{-2}}=\til[q]^{\frac{1}{2}},
\end{eqnarray*}
so that $\til[E_i],\til[F_i],\til[K_i]$ generates $\cU_{\til[q]}(\g_\R)$ of type $C_2$ where $i=1$ corresponds to the long root and $i=2$ corresponds to the short root.
\end{Lem}
\begin{proof} Let us choose the positive representation of $B_2$ corresponding to $w_0=s_1s_2s_1s_2$ given by Theorem \ref{B2}*. Then the transcendental relations for $E_2$ is clear from \eqref{E1},
$$([w]e(-p_w))^{\frac{1}{b^2}}=e^{\pi b\inv(w-2p_w)}+e^{\pi b\inv(-w-2p_w)}:=([w]e(-p_w))_*.$$

The transcendental relations for $E_1$ is less trivial. Let $q_s=q_1$ and write the terms in \eqref{e1B} as
\begin{eqnarray*}
e_1&=&[t]e(-p_t-p_u+p_w)+[u-v]e(-p_u-p_v+p_w)+[v-w]e(-p_v)\\
&=&(A_1^-+A_1^+)+(A_2^-+A_2^+)+(A_3^-+A_3^+)
\end{eqnarray*}
where each bracket corresponds to the two terms in $[-]e(-)$ so that 
$$A_i^+A_i^-=q_s^2A_i^-A_i^+.$$ 
Then by direct calculation we can write the above sum as
$$e_1=A_3^++A_2^++A_1^++A_1^-+A_2^-+A_3^-$$
so that each term $q_s^2$ commute with all other terms to their right, except for $A_3$ and $A_2$ where we have instead
$$A_3^+ A_2^+=q_s^{4}A_2^+ A_3^+,\tab A_2^- A_3^-= q_s^{4}A_3^-A_2^-,$$
Hence by Lemma \ref{qbi}, we obtain
\begin{eqnarray*}
(e_1)^{\frac{1}{b_s^2}}&=&(A_3^++A_2^++A_1^++A_1^-+A_2^-+A_3^-)^{\frac{1}{b_s^2}}\\
&=&(A_3^++A_2^+)^{\frac{1}{b_s^2}}+(A_1^++A_1^-+A_2^-+A_3^-)^{\frac{1}{b_s^2}}\\
&=&(A_3^++A_2^+)^{\frac{1}{b_s^2}}+(A_1^++A_1^-)^{\frac{1}{b_s^2}}+(A_2^-+A_3^-)^{\frac{1}{b_s^2}}\\
&=&({A_3^+}^{\frac{1}{2b_s^2}}+{A_2^+}^{\frac{1}{2b_s^2}})^2+({A_1^+}^{\frac{1}{b_s^2}}+{A_1^-}^{\frac{1}{b_s^2}})+({A_2^-}^{\frac{1}{2b_s^2}}+{A_3^-}^{\frac{1}{2b_s^2}})^2\\
&=&{A_3^+}^{\frac{1}{b_s^2}}+[2]_{\til[q_2]}\til[q_2]\inv A_3^+A_2^++{A_2^+}^{\frac{1}{b_s^2}}+{A_1^+}^{\frac{1}{b_s^2}}\\
&&+{A_3^-}^{\frac{1}{b_s^2}}+[2]_{\til[q_2]}\til[q_2] A_3^-A_2^-+{A_2^-}^{\frac{1}{b_s^2}}+{A_1^-}^{\frac{1}{b_s^2}}\\
&=&([2u-v]e(2p_w-p_v-2p_u))_*+[2]_{\til[q_2]}([u-w]e(p_w-p_v-p_u))_*\\
&&+([v-2w]e(-p_v))_*+[t]e(2p_w-2p_u-p_t))_*\nonumber\\
\end{eqnarray*}
Upon identification $(t,u,v,w)\corr (w,v,u,t)$ we see that this is exactly the expression \eqref{e2C} for the long root element $E_2$ of $C_2$ with $q$ replaced by $\til[q]$.

Finally, the operators $F_1,K_1\corr F_2,K_2$, $q_i\corr \til[q_i]$ under the transcendental relations. The transcendental relations of $F_i$ follow also from simple application of Lemma \ref{qbi} and the transcendental relations for $K_i$ are trivial.
\end{proof}

\begin{proof}(Theorem \ref{main}) In general, the transcendental relations for $F_i,K_i$ can be checked directly from Definition \ref{FH}*. Since each term of $f_i$ is $q_i^2$ commuting with each other, we can apply Lemma \ref{qbi} and induction to conclude that $f_i^{\frac{1}{b_i^2}}$ has exactly the same expression with $b_i$ replaced by $b_i\inv$. Hence under the scaling of the short root it is equivalent to the Cartan matrix $(a_{ij})$ being transposed.

For the action of $\til[E_i]$, it suffices to choose a simple reduced expression of $w_0$ in order to check the commutation relation, much like proof given in \cite{Ip}. In particular, we can choose the reduced expression of $w_0$ that ends with $s_is_js_i$ for simply-laced connecting root indices, or $s_is_js_is_j$ for doubly-connected root indices. Then we see from \eqref{E1} and Lemma \ref{B2trans} that the relations between $\til[E_i],\til[F_i],\til[K_i]$ and $\til[E_j],\til[F_j], \til[K_j]$ hold as generators of $\cU_{\til[q]}({}^L\g_\R)$.

Finally, the action $\til[E_1]$ and $\til[E_2]$ for $\cU_{\til[q]}(\g_\R)$ of type $G_2$ corresponding to $w_0=s_2s_1s_2s_1s_2s_1$ from Theorem \ref{G2} can be checked along the same line. One needs to apply appropriate quantum dilogarithm transformations using Lemma \ref{qsum} to simplify the expression, and apply the transcendental relation to the corresponding terms. Then reapply the inverse transformations will expand the expression and give the 13 terms which coincide with the action of $E_2$ and $E_1$ corresponding to $w_0=s_1s_2s_1s_2s_1s_2$ with $b_i$ replaced by $b_i\inv$. 
\end{proof}

Therefore we have the correspondence of generators:
\begin{eqnarray}
\begin{tikzpicture}[scale=.4,baseline=-1ex]
    \draw[xshift=0 cm,thick,fill=black] (0 cm, 0) circle (.3 cm);
    \foreach \x in {1,...,5}
    \draw[xshift=\x cm,thick] (\x cm,0) circle (.3cm);
    \draw[dotted,thick] (8.3 cm,0) -- +(1.4 cm,0);
    \foreach \y in {1.15,...,3.15}
    \draw[xshift=\y cm,thick] (\y cm,0) -- +(1.4 cm,0);
    \draw[thick] (0.3 cm, .1 cm) -- +(1.4 cm,0);
    \draw[thick] (0.3 cm, -.1 cm) -- +(1.4 cm,0);
    \foreach \z in {1,...,5}
    \node at (2*\z-2,-1) {$\z$};
\node at (10,-1){$n$};
\node at (5, -2){$B_n$};
  \end{tikzpicture}
&\longleftrightarrow&
  \begin{tikzpicture}[scale=.4,baseline=-1ex]
    \draw[xshift=0 cm,thick] (0 cm, 0) circle (.3 cm);
    \foreach \x in {1,...,5}
    \draw[xshift=\x cm,thick,fill=black] (\x cm,0) circle (.3cm);
    \draw[dotted,thick] (8.3 cm,0) -- +(1.4 cm,0);
    \foreach \y in {1.15,...,3.15}
    \draw[xshift=\y cm,thick] (\y cm,0) -- +(1.4 cm,0);
    \draw[thick] (0.3 cm, .1 cm) -- +(1.4 cm,0);
    \draw[thick] (0.3 cm, -.1 cm) -- +(1.4 cm,0);
    \foreach \z in {1,...,5}
    \node at (2*\z-2,-1) {$\z$};
\node at (10,-1){$n$};
\node at (5, -2){$C_n$};\node at (5, -3){};
  \end{tikzpicture}\\
  \begin{tikzpicture}[scale=.4,baseline=-1ex]
    \draw[thick] (-2 cm ,0) circle (.3 cm);
	\node at (-2,-1) {$1$};
    \draw[thick] (0 ,0) circle (.3 cm);
	\node at (0,-1) {$2$};
    \draw[thick,fill=black] (2 cm,0) circle (.3 cm);
	\node at (2,-1) {$3$};
    \draw[thick,fill=black] (4 cm,0) circle (.3 cm);
	\node at (4,-1) {$4$};
    \draw[thick] (15: 3mm) -- +(1.5 cm, 0);
    \draw[xshift=-2 cm,thick] (0: 3 mm) -- +(1.4 cm, 0);
    \draw[thick] (-15: 3 mm) -- +(1.5 cm, 0);
    \draw[xshift=2 cm,thick] (0: 3 mm) -- +(1.4 cm, 0);
\node at (1, -2){$F_4$};\node at (1, -3){};
  \end{tikzpicture} 
&\longleftrightarrow&
\begin{tikzpicture}[scale=.4,baseline=-1ex]
    \draw[thick,fill=black] (-2 cm ,0) circle (.3 cm);
	\node at (-2,-1) {$4$};
    \draw[thick,fill=black] (0 ,0) circle (.3 cm);
	\node at (0,-1) {$3$};
    \draw[thick] (2 cm,0) circle (.3 cm);
	\node at (2,-1) {$2$};
    \draw[thick] (4 cm,0) circle (.3 cm);
	\node at (4,-1) {$1$};
    \draw[thick] (15: 3mm) -- +(1.5 cm, 0);
    \draw[xshift=-2 cm,thick] (0: 3 mm) -- +(1.4 cm, 0);
    \draw[thick] (-15: 3 mm) -- +(1.5 cm, 0);
    \draw[xshift=2 cm,thick] (0: 3 mm) -- +(1.4 cm, 0);
\node at (1, -2){$F_4$};
  \end{tikzpicture}\\
 \begin{tikzpicture}[scale=.4,baseline=-1ex]
    \draw[thick] (0 ,0) circle (.3 cm);
    \draw[thick,fill=black] (2 cm,0) circle (.3 cm);
    \draw[thick] (30: 3mm) -- +(1.5 cm, 0);
    \draw[thick] (0: 3 mm) -- +(1.4 cm, 0);
    \draw[thick] (-30: 3 mm) -- +(1.5 cm, 0);
\node at (0,-1){$1$};
\node at (2,-1){$2$};
\node at (1, -2){$G_2$};\node at (1, -3){};
  \end{tikzpicture} 
&\longleftrightarrow&
  \begin{tikzpicture}[scale=.4,baseline=-1ex]
    \draw[thick,fill=black] (0 ,0) circle (.3 cm);
    \draw[thick] (2 cm,0) circle (.3 cm);
    \draw[thick] (30: 3mm) -- +(1.5 cm, 0);
    \draw[thick] (0: 3 mm) -- +(1.4 cm, 0);
    \draw[thick] (-30: 3 mm) -- +(1.5 cm, 0);
\node at (0,-1){$2$};
\node at (2,-1){$1$};
\node at (1, -2){$G_2$};
  \end{tikzpicture}
\end{eqnarray}

\section{Modified quantum group and its modular double}\label{sec:modified}
As in the simply-laced case, we note that the generators $E_i,F_i,K_i$ in general do not commute with $\til[E_i],\til[F_i],\til[K_i]$. For example, in type $B_2$, the relation $K_1E_2=q\inv E_2K_1$ implies
$$K_1\til[E_2]=q^{\frac{1}{b^2}}\til[E_2]K_1=-\til[E_2]K_1.$$
More precisely, we have:
\begin{Prop}
The generators $E_i,F_i,K_i$ commute with $\til[E_i],\til[F_i],\til[K_i]$ up to a sign. We have
\begin{eqnarray*}
\til[E_i]E_j&=&(-1)^{a_{ij}}E_j\til[E_i]\\
\til[E_i]K_j&=&(-1)^{a_{ij}}K_j\til[E_i]\\
\til[K_i]E_j&=&(-1)^{a_{ij}}E_j\til[K_i]
\end{eqnarray*} 
and similarly for $E_i$ replaced by $F_i$.
\end{Prop}

Hence as in \cite{Ip}, we modify the generators with powers of $K_i$ in order to take care of the commutation relation between the generators and the other part of the modular double. It turns out the modification is exactly the same even for non simply-laced type, and the corresponding results are as follows.

\begin{Prop} For each node $i$ in the Dynkin diagram, we assign a weight $n_i\in\{0,1\}$ such that $|n_i-n_j|=1$ if $i,j$ are connected in the diagram, so that $n_i$ alternates along the edges.

We define $\fq:=q^2=e^{2\pi ib^2}$ and
\Eq{\fq_i:=\case{q_i^{-2}&\mbox{if $n_i=0$,}\\ q_i^2&\mbox{if $n_i=1$,}}} where $q_i$ is defined by \eqref{qiBCF} and \eqref{qiG} as before, and we define the modified quantum generators by
\begin{eqnarray}
\bE_i&:=&q_i^{n_i} E_iK_i^{n_i},\\
\bF_i&:=&q_i^{1-n_i}F_iK_i^{n_i-1},\label{modified}\\
\bK_i&:=&\fq_i^{H_i}=\case{K_i^{-2}&\mbox{if $n_i=0$,}\\K_i^{2}&\mbox{if $n_i=1$,}}
\end{eqnarray}

Then the variables are positive self-adjoint. Let 
\Eq{[A,B]_\fq=AB-\fq\inv BA}
be the quantum commutator. Then the quantum relations in the new variables become:
\begin{eqnarray}
\label{bKE}\bK_i\bE_j&=&{\fq_i}^{a_{ij}}\bE_j\bK_i,\\
\label{bKF}\bK_i\bF_j&=&{\fq_i}^{-a_{ij}}\bF_j\bK_i,\\
\bE_i\bF_j&=&\bF_j\bE_i\tab \mbox{ if $i\neq j$},\\
{[\bE_i,\bF_i]}_{\fq_i}&=&\frac{1-\bK_i}{1-\fq_i},
\end{eqnarray}
 and the quantum Serre relations become
\begin{eqnarray}
[[[[\bE_j,\bE_{i}]_{\fq_i^{-a_{ij}}},...\bE_{i}]_{\fq_i^2},\bE_i]_{\fq_i},\bE_i]&=&0,\\
{[[[[\bF_j,\bF_{i}]_{\fq_i^{-a_{ij}}},...\bF_{i}]_{\fq_i^2},\bF_i]_{\fq_i},\bF_i]}&=&0.
\end{eqnarray}
We denote the modified quantum group by $\bU_{\fq}(\g_\R)$.
\end{Prop}

Hence we can now state the main theorem for $\bU_{\fq\til[\fq]}(\g_\R)$:

\begin{Thm}
Let $\til[\fq]:=\til[q]^2=e^{2\pi \bi b_s^{-2}}$. We define the tilde part of the modified modular double by 
\begin{eqnarray}
\til[\be_i]&:=&(\be_i)^{\frac{1}{b_i^2}},\\
\til[\bf_i]&:=&(\bf_i)^{\;\frac{1}{b_i^2}},\\
\til[\bK_i]&:=&(\bK_i)^{\frac{1}{b_i^2}},
\end{eqnarray}
where as before
\begin{eqnarray}
\be_i=2\sin(\pi b_i^2)\bE_i,&\tab&\bf_i=2\sin(\pi b_i^2)\bF_i,\\
\til[\be_i]=2\sin(\pi b_i^{-2})\til[\bE_i],&\tab&\til[\bf_i]=2\sin(\pi b_i^{-2})\til[\bF_i].
\end{eqnarray}

Then the properties of positive representations are satisfied:

\begin{itemize}
\item[(i)] the operators $\be_i,\bf_i,\bK_i$ and their tilde counterparts are represented by positive essentially self-adjoint operators,
\item[(ii)] the operators $\til[\bE_i],\til[\bF_i],\til[\bK_i]$ generates the Langlands dual $\bU_{\til[\fq]}({}^L\g_\R)$,
\item[(iii)] \mbox{all the generators $\bE_i,\bF_i,\bK_i$ commute with all $\til[\bE_j],\til[\bF_j],\til[\bK_j]$.}
\end{itemize}
 Therefore we have constructed the positive principal series representations of the modular double 
\Eq{\bU_{\fq\til[\fq]}(\g_\R):=\bU_{\fq}(\g_\R)\ox \bU_{\til[\fq]}({}^L\g_\R),}
parametrized by $rank(\g)$ numbers $\l_i\in\R$.
\end{Thm}

Not surprisingly, all the additional properties in the simply-laced case \cite{Ip} also hold in the general case, with slight modifications. The proof of the following result about the commutant is exactly the same as in the simply-laced case (with $A^T$ introduced due to the factor $\frac{1}{2}(\a_i,\a_i)$ in the proof).

\begin{Thm}\label{Langlands}The commutant for the positive representation of $\bU_{\fq}(\g_\R)$ is generated by $\til[\bE_i],\til[\bF_i]$ and elements of the form
\Eq{\label{KKKK}\til[\bK^{\bb^k}]:=\til[\bK]_1^{b_1^k}\til[\bK]_2^{b_2^k}\cdots\til[\bK]_{n}^{b_{n}^k},}
for $k=1,...,n$, where the vector $\bb_k=(b_1^k,b_2^k,...,b_{n}^k)^T$ satisfies
\Eq{A^T\bb_k=\be_k}
where $A=(a_{ij})$ is the Cartan matrix, $A^T$ its transpose, and $\be_k$ are the standard unit vectors. 
\end{Thm}

\begin{Rem} We note that the lattice generated by $\til[\bK_i]$ is a sub-lattice generated by $\til[\bK^{\bb^k}]$. In fact one can check that the commutant of (the adjoint form) $\bU_{\fq}(\g_\R)$ is precisely the simply-connected form of its Langlands dual quantum group, denoted by $\widehat{\bU_{\til[\fq]}}({}^L\g_\R)$. For example, in type $B_n$, $\widehat{\bU_{\til[\fq]}}({}^L\g_\R)$ is just adjoining $\til[K_1]^{\frac{1}{2}}$ to $\bU_{\til[\fq]}({}^L\g_\R)$.
\end{Rem}

Let the index $i$ corresponds to the coordinate parametrization $\{x_i\}$, and $b_i$ as usual denotes $b_s$ or $b_l$ depending on whether the corresponding root is short or long. Let $s$ (resp. $l$) be the number of indices corresponding to short (resp. long) roots, so that $s+l=l(w_0)$. Let $\C[\T_{\fq\til[\fq]}^{s,l}]$ be the quantum torus algebra generated by $\{\bu_i^{\pm1},\bv_i^{\pm1},\til[\bu_i]^{\pm1},\til[\bv_i]^{\pm1}\}_{i=1}^{l(w_0)}$ with the $\fq$-commutation relations:
\Eq{\bu_{i}\bv_{i}=\fq_i \bv_{i}\bu_{i}, \tab \til[\bu_{i}]\til[\bv_{i}]=\til[\fq_i]\til[\bv_{i}]\til[\bu_{i}],}
 which can be realized by positive self adjoint operators
\Eq{\label{fulltorus}\bu_{i}=e^{2\pi b_iu_{i}},\tab \bv_{i}=e^{2\pi b_ip_{i}},}
and similarly for $\til[\bu_{i}],\til[\bv_{i}]$ with $b_i$ replaced by $b_i\inv$. Then we have Theorem \ref{Thmqtori} of the introduction generalizing the results of the simply-laced case.

\begin{Thm} We have an embedding
\Eq{\bU_{\fq\til[\fq]}(\g_\R)\inj \C[\T_{\fq\til[\fq]}^{s,l}]}
where each generator of $\bU_{\fq\til[\fq]}(\g_\R)$ is realized as a Laurent polynomial of $\{\bu_i, \bv_j\}$.

In particular
\begin{eqnarray*}
\bU_{\fq\til[\fq]}(\g_\R)&\inj& \C[\T_{\fq\til[\fq]}^{n,n^2-n}]\tab\mbox{$\g$ is of type $B_n$}\\
\bU_{\fq\til[\fq]}(\g_\R)&\inj& \C[\T_{\fq\til[\fq]}^{n^2-n,n}]\tab\mbox{$\g$ is of type $C_n$}\\
\bU_{\fq\til[\fq]}(\g_\R)&\inj& \C[\T_{\fq\til[\fq]}^{12,12}]\;\;\;\tab\mbox{$\g$ is of type $F_4$}\\
\bU_{\fq\til[\fq]}(\g_\R)&\inj& \C[\T_{\fq\til[\fq]}^{3,3}]\tab\tab\mbox{$\g$ is of type $G_2$}
\end{eqnarray*}
\end{Thm}
\begin{proof} The proof is exactly the same as in the simply-laced case, with scaling by appropriate $b$'s. Since positive representations are irreducible, to obtain such embedding, all we need is a transformation of the representations of the modified generators $\bE_i, \bF_i, \bK_i$ that shifts the $p$'s appropriately so that the coefficients of the variables are all even. In particular, since only the parity matters, we observe that we can ignore the brackets in our notation (cf. Definition \ref{notation}) when considering such shifts. Note that this embedding is by no mean unique.

First we do this for the expressions of $\bF_i$. Explicitly, as before let $v_i$ denote the coordinate used in Definition \ref{FH}* and $r(i)$ its corresponding root. Also let $$c_{jk}:=a_{r(k)r(j)}b_k/b_j.$$ Then the unitary transformation can be obtained from multiplication by the unitary functions
\begin{eqnarray*}
e^{\frac{1}{2}\pi iv_i^2}&:&2\pi b_ip_i\mapsto 2\pi b_ip_i+\pi b_iv_i,
\end{eqnarray*}
for all $i$, and
\begin{eqnarray*}
e^{\pi ic_{jk}v_jv_k}&:&2\pi b_jp_j\mapsto 2\pi b_jp_j+\pi a_{r(k)r(j)}b_k v_k,\tab 2\pi b_kp_k\mapsto 2\pi b_kp_k+\pi a_{r(j)r(k)}b_jv_j
\end{eqnarray*}
whenever $j<k$, $n_{r(j)}=0$ and $r(j),r(k)$ are adjacent in the Dynkin diagram. Note that these transformations commute with each other, so the order does not matter.

One can check from Definition \ref{FH}* and the modified formula \eqref{modified} that the first shifts from these transformations systematically cancel the odd coefficients for the terms of $\bF_i$ with $n_i=0$. The second shift then deals with the odd coefficients in the terms of $\bF_j$ adjacent to $\bF_i$ in the Dynkin diagram, corresponding to $n_j=1$.

In the original Definition \ref{FH}*, since each term of $\bF_i$ contains only a single $e(p_{v_k})$, and commutes or $\fq_{r(k)}$-commutes with other terms in $\bE_j$, the transformations defined above will automatically force all the coefficients of the terms in $\bE_j$ to be even as well.
\end{proof}

Finally as in the simply-laced case, we have the following unitary equivalence of positive representations:
\begin{Thm}Let $\cP_\l$ denote the positive principal series representations of $\cU_q(\g_\R)$ corresponding to the parameter $\l=(\l_i)_{i=1}^n$ where $n=rank(\g)$. Then 
\Eq{\cP_\l\simeq \cP_{w(\l)}}
are unitary equivalent representations for any Weyl group element $w\in W$ acting on $\l$, namely for simple reflections,
\Eq{s_i(\l_j):=\l_j-a_{ij}\l_i}
where $a_{ij}$ is the Cartan matrix. In particular, the positive principal series representations are parametrized by $\l\in\R_{\geq 0}^n$.
\end{Thm}
\begin{proof} The proof is exactly the same as in the simply-laced case, slightly modified with the parameter $b$ replaced by the appropriate $b_i$'s when applying the intertwiner $G_{\l_i}$ (cf. (11.4) in \cite{Ip}).
\end{proof}


\end{document}